\documentclass {article}
\usepackage {geometry}

\usepackage [T1]{fontenc}
\usepackage [english]{babel}
\usepackage {mathbbol}
\usepackage {amsfonts}
\usepackage {amsthm}
\usepackage {amsopn}
\usepackage {mathtools}
\usepackage {tikz-cd}
\usepackage {hyperref}
\usepackage {csquotes}

\newtheorem{innercustomteo}{Theorem}
\newenvironment{customteo}[1]
  {\renewcommand\theinnercustomteo{#1}\innercustomteo}
  {\endinnercustomteo}

\newcommand{\ols}[1]{\mskip.5\thinmuskip\overline{\mskip-.5\thinmuskip {#1} \mskip-.5\thinmuskip}\mskip.5\thinmuskip} 
\newcommand{\olsi}[1]{\,\overline{\!{#1}}} 
\makeatletter
\newcommand\closure[1]{
  \tctestifnum{\count@stringtoks{#1}>1} 
  {\ols{#1}} 
  {\olsi{#1}} 
}
\long\def\count@stringtoks#1{\tc@earg\count@toks{\string#1}}
\long\def\count@toks#1{\the\numexpr-1\count@@toks#1.\tc@endcnt}
\long\def\count@@toks#1#2\tc@endcnt{+1\tc@ifempty{#2}{\relax}{\count@@toks#2\tc@endcnt}}
\def\tc@ifempty#1{\tc@testxifx{\expandafter\relax\detokenize{#1}\relax}}
\long\def\tc@earg#1#2{\expandafter#1\expandafter{#2}}
\long\def\tctestifnum#1{\tctestifcon{\ifnum#1\relax}}
\long\def\tctestifcon#1{#1\expandafter\tc@exfirst\else\expandafter\tc@exsecond\fi}
\long\def\tc@testxifx{\tc@earg\tctestifx}
\long\def\tctestifx#1{\tctestifcon{\ifx#1}}
\long\def\tc@exfirst#1#2{#1}
\long\def\tc@exsecond#1#2{#2}
\makeatother

\newcommand{\C}{{\mathbb{C}_\infty}}
\newcommand{\K}{{K_\infty}}
\newcommand{\F}{\mathbb{F}}
\newcommand{\T}{\mathbb{T}}
\newcommand{\Z}{\mathbb{Z}}
\renewcommand{\P}{\mathbb{P}}
\newcommand{\m}{\mathfrak{m}}
\newcommand{\A}{\mathcal{A}}
\renewcommand{\O}{\mathcal{O}}
\renewcommand{\L}{\mathcal{L}}
\DeclareMathOperator{\Hom}{Hom}
\DeclareMathOperator{\Pic}{Pic}
\DeclareMathOperator{\Div}{Div}
\DeclareMathOperator{\Spec}{Spec}
\DeclareMathOperator{\Frob}{Frob}
\DeclareMathOperator{\Q}{Frac}
\DeclareMathOperator{\red}{red}
\DeclareMathOperator{\sgn}{sgn}
\DeclareMathOperator{\G}{Gal}
\DeclareMathOperator{\Sf}{Sf}
\newcommand{\into}{\hookrightarrow}
\newtheorem{teo}{Theorem}[section]
\newtheorem*{teo*}{Theorem}
\newtheorem{lemma}[teo]{Lemma}
\newtheorem{prop}[teo]{Proposition}
\newtheorem*{prop*}{Proposition}
\newtheorem{cor}[teo]{Corollary}
\theoremstyle{definition}
\newtheorem{Def}[teo]{Definition}
\newtheorem{oss}[teo]{Remark}
\newtheorem*{oss*}{Remark}
\newtheorem{ex}[teo]{Example}
\newcommand{\catname}[1]{\mathbf{#1}}
\usepackage[
backend=biber,
style=alphabetic,
sorting=nyt
]{biblatex}
\title{A class of functional identities associated to curves over finite fields}
\author{Giacomo Hermes Ferraro\thanks{Dipartimento di Matematica Guido Castelnuovo, Università degli Studi di Roma "La Sapienza"}
}
\date{giacomohermes.ferraro@gmail.com}

\begin{document}
\maketitle

\begin{abstract}
Goss zeta values can be found, in some cases, as evaluations of a new type of rigid analytic function on projective curves $X$ over a finite field $\F_q$, called ``Pellarin $L$-series". In the case of genus $0$ and $1$, Pellarin and Green--Papanikolas further determined functional identities for Pellarin $L$-series, in partial analogy with the functional equation of Dirichlet $L$-series.

The aim of this paper is to prove that a generalization of these functional identities holds in arbitrary genus. Our proof exploits the topological nature of divisors on the curve $X$, as well as the introduction of an ``adjoint shtuka function". This allows us to reinterpret Pellarin $L$-series as dual versions of the special functions studied by Angl\`es, Ngo Dac, and Tavares Ribeiro.
\end{abstract}

\section{Introduction}

\subsection{Historical premises and motivation}

Let $\F_q$ be the finite field with $q$ elements. The Carlitz module over the field of rational functions $\F_q(\theta)$ is a functor from the $A$-algebras to the $A$-modules $C:A-\catname{Alg}\to A-\catname{Mod}$, where $A=\F_q[\theta]$, which is meant to be a finite characteristic analogue to the multiplicative group scheme $\mathbb{G}_m:\Z-\catname{Alg}\to\Z-\catname{Mod}$. For any $\F_q[\theta]$-algebra $S$, $C(S)$ is $S$ with a new $\F_q[\theta]$-module structure uniquely determined by the action of $\theta$: $C_\theta(s):=s^q+\theta s$ for all $s\in S$.
\vspace{5mm}

\noindent
\textbf{Gauss--Thakur sums.} Let's denote by $\C$ the completion of an algebraic closure of $\F_q((\theta^{-1}))$: this is a complete algebraically closed field analogous to the field of complex numbers $\mathbb{C}$. Analogously to the classical case, there is a surjective map of $A$-modules $\exp_C:\C\to C(\C)$, with kernel $\tilde{\pi}A\subseteq\C$ for a certain $\tilde{\pi}\in\mathbb{C}_\infty^\times$. Given a nonzero $r\in A$ and a multiplicative character $\chi:(A/rA)^\times\to\closure{\F_q}^\times$, we can define the Gauss--Thakur sum introduced by Thakur in \cite{Thakur88} (in the generality proposed by Gazda and Maurischat in \cite{Gazda}) as an element $G(\chi)\in\C\otimes_{\F_q} \closure{\F_q}$ with the property that $(C_\theta\otimes1-1\otimes\chi(\theta))G(\chi)=0$.
If $\chi$ cannot be lifted to an $\F_q$-algebra homomorphism $A/rA\to\closure{\F_q}$, $G(\chi)=0$ (\cite{Gazda}[Prop. 4.8]); this suggests the existence of a ``universal" Gauss--Thakur sum, independent from $\chi$, whose specialization yields all possible nonzero Gauss--Thakur sums. Pellarin and Angl\`es showed in \cite{AP} that the ``universal" Gauss--Thakur sum is the special function of Anderson--Thakur $\omega=\sum_i c_i t^i\in\C[[t]]$, first introduced by Anderson and Thakur in \cite{AndersonThakur}, with the property that $c_0=1$ and $C_\theta(c_{i+1})=c_i$ for all $i$. It turns out that $\omega$ belongs to the Tate algebra $\C\langle t\rangle$, i.e. it can be evaluated at $t=\xi$ for any $\xi\in\closure{\F_q}$, and that every such evaluation $\omega(\xi)$, as an element of $\C\otimes_{\F_q}\closure{\F_q}$, is a Gauss--Thakur sum (see also \cite{Gazda}[Lemma 4.5]).
\vspace{5mm}

\noindent
\textbf{Pellarin $L$-values.} In analogy with the Riemann zeta function, for any positive integer $s$ the infinite series $\sum h^{-s}\in\C$, where we sum over all monic $h\in A$, is well defined. For any nonzero $r\in A$ and any multiplicative character $\chi:(A/rA)^\times\to\closure{\F_q}^\times$, we can also define the Dirichlet-like $L$-value at $s$ as $L(\chi,s):=\sum h^{-s}\otimes \chi(h)\in\C\otimes_{\F_q}\closure{\F_q}$, following the same principle that Gazda and Maurischat adopted for their generalization of Gauss--Thakur sums. In a seminal paper (\cite{Pellarin2011}), Pellarin introduced the series
\[L(s):=\sum_{\substack{h\in \F_q[t]\\\text{monic}}} \frac{h(t)}{h(\theta)^s}\in\C[[t]],\]
where $s$ is a positive integer. For any $\F_q$-algebra homomorphism $\chi:\F_q[t]\to\closure{\F_q}$, $L(s)$ can be evaluated at $\chi(t)$, obtaining $L(\chi,s)$. In other words, the Pellarin zeta function interpolates (some) Dirichlet-like series $L(\chi,s)$, in the same way as the Anderson--Thakur special function interpolates Gauss--Thakur sums.

\vspace{5mm}

\noindent
\textbf{The functional identity.} The Riemann zeta function (and Dirichlet $L$-functions in general) have a well-known functional equation. The formulation of an analogue for our Dirichlet-like series is still an open problem (see \cite{Goss}[Subsection 8.1]), so let's fix a multiplicative character $\rho:\Z/n\Z^\times\to\mathbb{C}^\times$ with $\rho(-1)=-1$ and consider the evaluation of the functional equation for the classical Dirichlet $L$-function $\L(\closure\rho,s)$ at $s=0$:
\begin{equation}\label{classical eq.}
    \frac{i\pi}{n}\sum_{j\in\Z/n\Z^\times}\rho(j)^{-1}j=\mathcal{G}(\rho)\L(\rho,1),
\end{equation}
where $\mathcal{G}(\rho)=\sum_{j\in\Z/n\Z^\times}\rho(j)^{-1}e^{\frac{2\pi i j}{n}}$ is a Gauss sum, and the left hand side is the explicit formula for $\L(\closure\rho,0)$. Going back to the function field case, in \cite{Pellarin2011}[Thm. 1], using two different methods (one involving modular forms, the other one using some log-algebraicity results of Anderson from \cite{Anderson} and \cite{Anderson2}) Pellarin proved the following identity in $\C\langle t\rangle$:
\begin{equation}\label{eq. Pellarin}
    \frac{\tilde{\pi}}{t-\theta}=\omega L(1).
\end{equation}
If we specialize this identity at an $\F_q$-linear homomorphism $\chi:\F_q[t]\to\closure{\F_q}$ sending $t$ to some $n$-th root of unity, we deduce:
\begin{equation}\label{eval. eq. Pellarin}
    \frac{\tilde{\pi}}{\theta-\theta^{n+1}}\sum_{j=1}^{n}\chi(t)^{-j}\theta^j=G(\chi)L(\chi,1),
\end{equation}
which is remarkably similar to equation (\ref{classical eq.}). In practice, Pellarin's identity (\ref{eq. Pellarin}) interpolates the equation (\ref{eval. eq. Pellarin}) across all suitable characters for the fixed exponent $s=1$, instead of interpolating across all suitable exponents for a fixed character $\chi$, like what happens in the classical setting.

\subsection{Generalization to extensions of \texorpdfstring{$\F_q[\theta]$}{Fq[t]}}\label{subsection Drinfeld}
The constructions of the previous subsection can be generalized to certain finite ring extensions of $\F_q[\theta]$. The aim of this article is to prove a generalization of Pellarin's identity (\ref{eq. Pellarin}) to such rings.

The following notation will be adopted from now on in the article. Let $X$ be a projective, geometrically irreducible, smooth curve of genus $g$ over $\F_q$, with a point $\infty\in X(\F_q)$. We now denote by $A$ the global sections of $X\setminus\{\infty\}$, by $\K$ the completion of $K:=\Q(A)$ at $\infty$, and by $\C$ the completion of an algebraic closure $\closure\K$. We also fix a uniformizer at $\infty\in X$, and denote by $\sgn$ the associated sign function.

A {\em normalized Drinfeld module} $\phi$ is a generalization of the Carlitz module $C$ to the ring $A$ (see Section \ref{notation} for details). Following \cite{Thakur}, we associate to $\phi$ a rational function $f$ on $X\times\Spec(\K)$ called \emph{shtuka function}, with $\sgn(f)=1$ and $\Div(f)=V^{(1)}-V+\Xi-\infty$ where $\Xi:\Spec(\C)\to\Spec(A)\into X$ is given by the canonical inclusion $A\into\C$, $V$ is the so-called \emph{Drinfeld divisor}, an effective divisor of degree $g$ on $X\times\Spec(\K)$, and $V^{(1)}$ denotes the \emph{Frobenius twist} of $V$ (see Definition \ref{def frobenius} and Remark \ref{oss twist div})

Instead of the Tate algebra $\C\langle t\rangle$, we define $\T:=\C\hat\otimes A$, i.e. the completion of the tensor product $\C\otimes_{\F_q} A$ with respect to the Gauss norm induced by the norm on $\C$, endowed with the $A$-module structure induced by $\phi$ on $\C\cong\phi(\C)$. The set of ``\emph{special functions}" $\Sf(\phi):=\{\omega\in\T|\phi_a(\omega)=(1\otimes a)\omega\;\forall a\in A\}$, introduced in \cite{ANDTR} by Angl\`es, Ngo Dac, and Tavares Ribeiro, generalizes the Anderson--Thakur function. On the other hand, the \emph{partial zeta function} $\zeta_I$ relative to a nonzero ideal $I<A$ is defined as the the series $\sum a^{-1}\otimes a\in\T$, where the sum ranges over the nonzero elements of $I$; $\zeta_I$ generalizes the Pellarin $L$-value $L(1)$, and like $L(1)$, it is a rigid analytic function on
the analytification of the affine curve $(X\setminus\{\infty\})\times\Spec(\C)$ (see \cite{Pellarin2021}).
\vspace{5mm}

\noindent
\textbf{Old and new results.} Green and Papanikolas were the first to generalize the functional identity (\ref{eq. Pellarin}) to a wider class of Drinfeld modules: their result (\cite{Green}[Thm. 7.1]) holds for all elliptic curves, i.e. $g(X)=1$. Together with Pellarin's identity, their theorem can be stated as follows.
\begin{teo*}[Green--Papanikolas]
    Suppose $g(X)\leq1$ and fix the unique normalized Drinfeld module $\phi$ with period lattice $\tilde{\pi}A\subseteq\C$ for some $\tilde{\pi}\in\mathbb{C}_\infty^\times$. There is a rational function $h\in H(X)^\times$ with $\sgn(h)=1$ such that:
    \[\Sf(\phi)=\frac{(\tilde{\pi}\otimes1)h}{\zeta_A}\cdot (\F_q\otimes A).\]
\end{teo*}
In the present paper, we prove the following generalization for any datum $(X,\infty,\sgn)$ with $X$ of arbitrary genus and $\infty\in X(\F_q)$. Let's fix an ideal $J<A$ such that $J^{-1}$ is isomorphic to the K\"ahler module of differentials $\Omega$, and let's denote by $H$ the Hilbert class field of $K$.
\begin{customteo}{A}[Theorem \ref{Sf module}]\label{A}
Let $\phi$ be a normalized Drinfeld module of rank $1$ with period lattice $\rho_I I$, where $\rho_I\in\mathbb{C}_\infty^\times$ and $I<A$ is a nonzero ideal. There is a rational function $h\in H(X)^\times$ with $\sgn(h)=1$ such that:
\[\Sf(\phi)=\frac{(\rho_I\otimes1)h}{\zeta_I}(\F_q\otimes IJ).\]
\end{customteo}
\begin{oss*}
    The element $\rho_I\in \C$ is uniquely determined up to a factor in $\F_q^\times$. Moreover, one can check that the right hand side only depends on the ideal classes of $I$ and $J$.
\end{oss*}
As with \cite{Green}[Thm. 7.1], the divisor of $h$ can be made explicit. If we denote by $d_J$ the degree of $J$ and we call $V_*$ the unique effective divisor of degree $g$ such that the divisor $V_*^{(1)}+V-J-\Xi-(2g-d_J-1)\infty$ is principal (see Proposition \ref{con} and Lemma \ref{lemma h_I,J}), that is exactly the divisor of $h$.

It's worth noting that the techniques employed by Green and Papanikolas were tailored to the case $g(X)=1$ -- for example, they used the Weierstrass model of a generic elliptic curve to carry out explicit computations. In the present paper, some computations are only needed in Section \ref{section duality} to find the scalar factor $\rho_I\otimes 1$;  in Section \ref{section special functions}, we prove Theorem \ref{Sf module weak} (a weak version of Theorem \ref{A}, without the explicit constant $\rho_I$) using the purely theoretical results of Section \ref{section topology} and Section \ref{Frobenius and divisors}.

Let's denote by $\cdot^{(1)}:\T\to\T$ the $A$-linear extension of the Frobenius endomorphism $\C\to\C$. If $\phi$ is a normalized Drinfeld module of rank $1$ and $f$ is its shtuka function, it is known that an element $\omega\in\T$ is a special function if and only if $\omega^{(1)}=f\omega$ (see \cite{ANDTR}[Lemma 3.6] and \cite{ANDTR}[Rmk. 3.10]). If $f\in\T^\times$ it's not difficult to construct an invertible special function $\omega\in\T^\times$ as an infinite product; Gazda and Maurischat noticed in \cite{Gazda}[Cor. 3.22] that, if there is an invertible special function, then $\Sf(\phi)\cong A$, leaving as an open question if the converse is true, or in other words how restrictive is the hypothesis that $f\in\T^\times$. In this paper we circumvent this problem entirely and prove that it is possible to construct a special function as an infinite product without assuming that $f\in\T^\times$.
\begin{customteo}{B}[Theorem \ref{omega infinite product}]\label{B}
    Fix a normalized Drinfeld module $\phi$ of rank $1$ with shtuka function $f$. There is some $\alpha\in K_\infty^\times$ such that the following element of $\C\hat\otimes K$ is well defined (up to the choice of a $q-1$-th root of $\alpha$):
\[\omega:=(\alpha\otimes1)^\frac{1}{q-1}\prod_{i\geq0}\left(\frac{\alpha\otimes1}{f}\right)^{(i)}.\] 
Moreover, $\omega\in(\F_q\otimes K)\Sf(\phi)$.
\end{customteo}
In Section \ref{section zeta functions} we introduce the \emph{adjoint shtuka function} $f_*$ relative to a normalized Drinfeld module $\phi$ of rank $1$. It is defined as the unique rational function on $X_\K$ with divisor $V_*-V_*^{(1)}+\Xi-\infty$ and $\sgn(f_*)=1$. In analogy with Theorem \ref{B}, we prove the following identity.
\begin{customteo}{C}[Theorem \ref{functional identity}]\label{C}
    Let $\rho_I I$ be the period lattice of $\phi$, where $I<A$ is a nonzero ideal and $\rho_I\in\mathbb{C}_\infty^\times$, and fix $a_I\in I$ an element of least degree. The following functional identity holds in $\C\hat\otimes K$:
    \[\zeta_I=-(a_I^{-1}\otimes a_I)\prod_{i\geq0}\left((\rho_I a_I\otimes1)^{1-q}f_*^{(1)}\right)^{(i)}.\]
\end{customteo}
Equivalently, we prove the following identity (Proposition \ref{a_I/pi}):
\[\frac{\left((\rho_I^{-1}\otimes1)\zeta_I\right)^{(-1)}}{(\rho_I^{-1}\otimes1)\zeta_I}=f_*.\]

Its similarity with the identity $\omega^{(1)}=f\omega$ defining a special function $\omega$ suggests this as an alternative definition of Pellarin zeta functions. Indeed we can use Proposition \ref{a_I/pi} to deduce the following theorem.

\begin{customteo}{D}[Theorem \ref{dual special function}]
Let $\rho_I I$ be the period lattice of $\phi$, where $I<A$ is a nonzero ideal and $\rho_I\in\mathbb{C}_\infty^\times$, and let $\phi^*:A\to\C[\tau^{-1}]$ denote the adjoint Drinfeld module. Then for all $a\in A$ we have the following identity:
\[\phi^*_a\left((\rho_I^{-1}\otimes1)\zeta_I\right)=(1\otimes a)(\rho_I^{-1}\otimes1)\zeta_I.\]
\end{customteo}

A weak version of Theorem \ref{C} (Theorem \ref{functional identity weak}, without the explicit constant $\rho_I$) is proven in Section \ref{section zeta functions}, while the strong version is contained in Section \ref{section duality}, where we also explore some aspects of the dual nature of special functions and zeta functions.

The structure of the rest of the paper is as follows. In Section \ref{section topology} we describe a functorial way of assigning a compact topology to the $\K$-points of a proper $\F_q$-scheme $Y$. We then discuss some results about divisors of curves in finite characteristic from \cite{Milne}, and deduce a homeomorphism between certain spaces of rational functions and the spaces of their divisors; this allows us to prove statements about the convergence of the former by looking at the latter. Green and Papanikolas had already conjectured that the Jacobian variety and the Drinfeld divisor $V$ would play a role in the generalization of \cite{Green}[Thm. 7.1], and in Section \ref{section topology} and Section \ref{Frobenius and divisors} we explain concretely how they are used.

Finally, in Section \ref{Anderson zeta}, we generalize another theorem of Green and Papanikolas about zeta functions ``\`a la Anderson" (\cite[Thm. 7.3]{Green}), which they used to prove a particular case of a log-algebraicity theorem by Anderson (\cite{Anderson}[Thm. 5.1.1]); the method does not differ substantially from the one used by Green and Papanikolas in \cite{Green}, but it heavily relies on our previous results. Given a normalized Drinfeld module $\phi$ of rank $1$, let $\rho_I I$ be the period lattice of $\phi$, where $I<A$ is a nonzero ideal and $\rho_I\in\mathbb{C}_\infty^\times$, and fix $a_I\in I$ an element of least degree. For any nonzero ideal $J=(a,b)<A$ we follow Hayes' construction in \cite{Hayes} and define $\phi_J\in H[\tau]$ as the monic generator of the left ideal $(\phi_a,\phi_b)<H[\tau]$; we fix a special function $\omega\in \Sf(\phi)$ and define $\chi(J):=\frac{\phi_J(\omega)}{\omega}$, a rational function on $X_H$ which does not depend on the choice of $\omega$. The Anderson zeta function relative to $\phi$ is defined as the following absolutely convergent series in $\K\hat\otimes K$ (see also Definition \ref{def anderson} and Remark \ref{oss anderson}):
\[\xi_\phi:=(a_I\otimes a_I^{-1})\sum_{J<I}\frac{\chi(I)}{\chi(I)(\Xi)}.\]
We prove the following result, where $\omega\in\C\hat\otimes K$ is defined as in Definition \ref{pseudocanonical special function}.
\begin{customteo}{E}[Theorem \ref{xi equation}]
    Fix an ideal $J<A$ of degree $d_J$ such that $J\Omega\cong A$, and denote by $h\in H(X)$ the unique rational function with sign $1$ and divisor equal to $V_*^{(1)}+V-J-\Xi-(2g-d_J-1)\infty$.
    There exists $\omega\in\Sf(\phi)\cdot(\F_q\otimes K)$ such that the following identity holds in $\C\hat\otimes K$:
\[\xi_\phi\omega=(\rho_I a_I\otimes1)\sum_{\sigma\in\G(H/K)}h^\sigma.\]
\end{customteo}

, by $\Omega$ the module of K\"ahler differentials of $A$, by $H$ the Hilbert class field of $K:=\Q(A)$,
For any field $L/\F_q$, we also denote by $X_L$ the base change $X\times\Spec(L)$ and by $L(X)$ the field of rational functions on $X_L$. We fix an inclusion $H\subseteq\K$ and a sign function at $\infty$, i.e. a multiplicative function from $\C(X)^\times$ to $\mathbb{C}_\infty^\times$ (see Remark \ref{oss sign}). Finally, given a rational divisor $D$ on $X_\C$, we denote by $D^{(1)}$ its .

. The natural action of $\C[\tau]$ on $\C$ is extended $A$-linearly to $\T$, and for any $h\in\T$ we define $h^{(n)}:=\tau^n h$; $\C[\tau]$ acts similarly on $\C(X)$, and we adopt the same notation (see also Definitions \ref{def frobenius} and \ref{def twist fun}, and Remark \ref{oss twist fun}).

\section{Notation and fundamental concepts}\label{notation}

Recall the meaning of $X,g,\infty,A,K,H,\K,\C$ from Subsection \ref{subsection Drinfeld}. 
Throughout this paper, it's appropriate to assume $g\geq1$. On one hand, all the proofs also work in the hypothesis $g=0$ with some caveats; on the other hand, this assumption simplifies the language of the paper by avoiding a fringe case, which was already studied in \cite{Pellarin2011}.
We introduce the following additional notation.
\begin{itemize}
    \item[-]The degree map $\deg:K\to\Z$ is defined as the opposite of the valuation at $\infty$, and for all $A$-modules $\Lambda\subseteq K$, for all integers $d$, we define $\Lambda(d)$ (resp. $\Lambda(\leq d)$) the set $\{x\in\Lambda|\deg(x)=d\}$ (resp. $\deg(x)\leq d$).
    \item[-]For any finite field extension $L/\K$, we denote by $(\O_L,\m_L)$ the associated local ring of integers.
    \item[-]Unlabeled tensor products of modules are assumed to be over $\F_q$, while for unlabeled fiber products of schemes the base ring is clear from the context.
    \item[-]If $Y$ is an $R$-scheme and $S$ is an $R$-algebra, we denote by $Y(S)$ the set of morphism of $R$-schemes from $\Spec(S)$ to $Y$, and by $Y_S$ the base change $Y\times\Spec(S)$. If $Y_S$ is integral, we denote by $S(Y)$ the field of rational functions of $Y_S$.
    \item[-]For all complete normed fields $L$, for all $\F_q$-vector spaces $M$, the module $L\otimes M$ is endowed with the sup norm induced by $L$; we denote its completion by $L\hat\otimes M$.
\end{itemize}
\begin{oss}\label{oss sign}
    Let's describe explicitly the sign function for any $h\in\C(X)^\times$. Since $\C(X)$ is the field of fractions of $\C\otimes A$, and the sign is multiplicative, we can assume $h\in\C\otimes A\setminus\{0\}$. We fix a uniformizer $w$ at $\infty\in X_\C$, which induces an immersion $i:\C\otimes A\into\C((w))$: we define $\sgn(h)$ as the leading coefficient of $i(h)\in\C((w))$.
    
    Note that, $\sgn(a)\in\F_q$ for all $a\in A\setminus\{0\}$, and if we write $h=\sum_{i=0}^k c_i\otimes a_i$, with $(a_i)_i$ in $A$ of strictly increasing degree and $(c_i)_i$ in $\mathbb{C}_\infty^\times$, we have:
    \[\sgn(h)=\sgn\left(\sum_{i=0}^k c_i\otimes a_i\right)=\sgn(c_k\otimes a_k)=c_k\sgn(a_k).\]
\end{oss}

In the rest of this section, we present some basic tools following \cite{Goss}.
Let $\C[\tau]$ and $\C[\tau^{-1}]$ be the rings of non-commutative polynomials over $\C$ with the relations $\tau c=c^q\tau$ and $\tau^{-1}c^q=c\tau^{-1}$ for all $c\in\C$. There is a $\F_q$-linear and bijective antihomomorphism $\C[\tau]\to\C[\tau^{-1}]$ sending $\varphi:=\sum_i c_i\tau^i$ to $\varphi^*:=\sum \tau^{-i} c_i$.

A \emph{Drinfeld module} of rank $r$ is a ring homomorphism $\phi:A\to\C[\tau]$ sending $a$ to $\phi_a:=\sum_{i\geq0}a_i\tau^i$ such that, for all $a\in A\setminus\{0\}$:
\[\deg_\tau(\phi_a)=r\deg(a)\text{ and }a_0=a.\]
If moreover $a_{\deg(a)}=\sgn(a)$ for all $a\in A\setminus\{0\}$, we call $\phi$ a \emph{normalized Drinfeld module}.

Fix $\phi,\psi$ Drinfeld modules. An element $f=\sum_i c_i \tau^i\in\C[\tau]$ is said to be an isogeny from $\phi$ to $\psi$ if $f\circ\phi_a=\psi_a\circ f$ for all $a\in A$, and an invertible isogeny is called isomorphism. It is known that every Drinfeld module is isomorphic to a (unique) normalized Drinfeld module.

The \emph{exponential map} relative to a discrete $A$-module $\Lambda\subseteq\C$ is the following analytic function from $\C$ to itself:
\[\exp_\Lambda(x):=x\prod_{\lambda\in\Lambda\setminus\{0\}}\left(1-\frac{x}{\lambda}\right)\in\C[[x]].\]
We can rearrange and write $\exp_\Lambda(x)=\sum_{i\geq0}e_ix^{q^i}$, which coverges for all $x\in\C$; its (bilateral) compositional inverse exists in $\C[[x]]$, is denoted by $\log_\Lambda=\sum_{i\geq0}l_i x^{q^i}$, and is called \emph{logarithmic map}. In particular, both functions can be seen as elements of the noncommutative power series ring $\C[[\tau]]$ ($\exp_\Lambda=\sum_{i\geq0}e_i\tau^i$ and $\log_\Lambda=\sum_{i\geq0}l_i\tau^i$), in which they are inverse to one another.

Finally, we recall that for any Drinfeld module $\phi$ of rank $1$ there is a unique $A$-submodule $\Lambda\subseteq\C$ of rank $1$, which we call the \emph{period lattice} of $\phi$, such that we have the following identity in $\C[[\tau]]$ for all $a\in A$:\[\exp_{\Lambda}\circ(a\tau^0)=\phi_a\circ\exp_{\Lambda}.\]Moreover, this correspondence is bijective, and isomorphic period lattices correspond to isomorphic Drinfeld modules.

\section{Convergence of divisors of rational functions on \texorpdfstring{$X_\K$}{X}}\label{section topology}

Since this section has the fundamental purpose of establishing some useful basic tools, its setting is slightly more general, in that we assume the closed point $\infty\in X$ to be $\F_{q^e}$-rational for some positive integer $e$.

Consider the $d$-th symmetric power $X^{[d]}$ for some positive integer $d$. For all field extensions $L/\F_q$, $X^{[d]}(L)$ is the set of $L$-rational divisors on $X$ of degree $d$. For any effective divisor $D\in X^{[d]}(\F_{q^e})$, for any finite field extension $L/\K$, we endow the space of global sections $H^0(X_L,D)\cong L\otimes_{\F_{q^e}}H^0(X_{\F_{q^e}},D)$ with the natural topology of finite vector space over $L$. The aim of this section is to endow $X^{[d]}(L)$ with a compact topology (see Definition \ref{compact topology}) such that the following proposition holds.

\begin{prop*}[Prop. \ref{convergence of functions and divisors}]

Fix a finite field extension $L/\K$ and an effective divisor $D_-$ in $X^{[d]}(\F_{q^e})$, and consider a sequence $(h_m)_m$ in $H^0(X_L,D_-)$.

If the sequence $(\Div(h_m)+D_-)_m$ converges to $D_+\in X^{[d]}(L)$ in the compact topology, there are $(\lambda_m)_m$ in $L^\times$ such that the sequence $(\lambda_m h_m)_m$ converges in $H^0(X_L,D_-)$ to some nonzero $h$ with $\Div(h)=D_+-D_-$.

If the sequence $(h_m)_m$ converges in $H^0(X_L,D_-)$ to some nonzero $h$, the sequence $(\Div(h_m)+D_-)_m$ converges to $\Div(h)+D_-\in X^{[d]}(L)$ in the compact topology.
\end{prop*}

In the following sections we need a topology on the $L$-points of other projective $\F_q$-schemes (such as the powers $\{X^d\}_{d\geq1}$ and the Jacobian variety $\A$ of $X$). To ensure their good interaction we prove that the compact topology that we define is functorial in Proposition \ref{Frobenius functor}.

\subsection{Functorial compact topology on \texorpdfstring{$\K$}{K}-rational points of \texorpdfstring{$\F_q$}{Fq}-schemes}\label{topology}

Through this subsection, $L$ is a finite field extension of $\K$ with residue field $\F_L\subseteq\O_L$, and $Y$ is a proper $\O_L$-scheme. We aim to construct a functor from proper schemes over $\O_L$ to compact Hausdorff topological spaces, sending $Y$ to $Y(\O_L)=Y(L)$.

\begin{lemma}\label{limit points}
The natural maps $\left(\red_{L,k}:Y(\O_L)\to Y(\O_L/\m_L^k)\right)_{k\geq1}$ induce a bijection $Y(\O_L)\cong\varprojlim_k Y(\O_L/\m_L^k)$.
\end{lemma}
\begin{proof}
Since $\Spec(\O_L)\cong\varinjlim_k\Spec(\O_L/\m_L^k)$, we have:
\begin{align*}
    Y(\O_L)&\cong\Hom_{\O_L}\left(\varinjlim_k\Spec(\O_L/\m_L^k),Y\right)\\
    &\cong\varprojlim_k\Hom_{\O_L}\left(\Spec(\O_L/\m_L^k),Y\right)\cong\varprojlim_k Y(\O_L/\m_L^k).\tag*{\qedhere}
\end{align*}
\end{proof}
\begin{oss}\label{limit topology}
If we endow the spaces $\{Y(\O_L/\m_L^k)\}_k$ with the discrete topology, the limit topology induced on $Y(\O_L)\cong\varprojlim_k Y(\O_L/\m_L^k)$ is Hausdorff. Since $Y$ is finite-type over $\O_L$ and $\O_L/\m_L^k$ is finite for all $k$, $Y(\O_L/\m_L^k)$ is finite for all $k$, so the limit topology makes $Y(\O_L)$ a compact space. Moreover, $Y(\O_L)$ can be endowed with an ultrametric distance $\bar{d}$ as follows:
\[\bar{d}(P,Q):=\max_{k\in\mathbb{N}}\left\{\frac{1}{p^k}\bigg|\red_{L,k}(P)\neq\red_{L,k}(Q)\right\}.\]
\end{oss}
\begin{Def}\label{compact topology}We call \emph{compact topology} the topology induced on $Y(L)=Y(\O_L)$ by the bijection $Y(\O_L)\cong\varprojlim_k Y(\O_L/\m_L^k)$.\end{Def}
\begin{Def}\label{red}
We denote by $\red_L:Y(\O_L)\to Y(\O_L)$ and call \emph{reduction} the map induced by the morphism $\O_L\to\F_L\subseteq\O_L$
\end{Def}

From this point onwards, unless otherwise stated, we interpret the set $Y(L)$ as endowed with the compact topology. Similarly, if $Y'$ is a proper $\F_q$-scheme, the set $Y'(L)=Y'_{\O_L}(L)$ is always endowed with the compact topology.

\begin{prop}\label{Frobenius functor}
The map associating to a proper $\O_L$-scheme $Y$ the topological space $Y(\O_L)$ can be extended to a functor $F_L$.
\end{prop}  
\begin{proof}
For every morphism $\varphi:Z\to Y$ of proper $\O_L$-schemes, the induced map $\varphi_{\O_L}:Z(\O_L)\to Y(\O_L)$ induces a system of maps $(\varphi_{\O_L/\m_L^k}:Z(\O_L/\m_L^k)\to Y(\O_L/\m_L^k))_k$ which commute with the transition maps of the diagrams $(Z(\O_L/\m_L^k))_k$ and $(Y(\O_L/\m_L^k))_k$, hence $\varphi_{\O_L}$ is continuous.

If we set $F_L(\varphi):=\varphi_{\O_L}$ for all morphisms, it's easy to check that $F_L$ sends the identity map to the identity map and preserves composition, hence it is a functor. 
\end{proof}
\begin{oss}\label{functor for F_p-schemes}
    We also obtain a functor from proper $\F_q$-schemes to topological spaces, sending a scheme $Y$ to $Y(\O_L)=Y(L)$, by precomposing $F_L$ with the base change $Y\mapsto Y_{\O_L}$.
\end{oss}

\begin{lemma}\label{local homeomorphism}
Let $f:Z\to Y$ be a morphism of proper $\O_L$-schemes. Fix a subset $V\subseteq Y(L)$ with preimage $U\subseteq Z(L)$, such that $F_L(f)|_U:U\to V$ is bijective. Then $F_L(f)|_U$ is a homeomorphism.
\end{lemma}
\begin{proof}
The map $F_L(f):Z(L)\to Y(L)$ is closed, being a continuous map between compact Hausdorff spaces.
Any closed set of $U$ can be written as $C\cap U$, with $C\subseteq Z(L)$ closed. We have:
\[F_L(f)(C\cap U)=F_L(f)\left(C\cap F_L(f)^{-1}(V)\right)=F_L(f)(C)\cap V,\]
which is closed in $V$ because $F_L(f)(C)$ is closed in $Y(L)$. This means that $F_L(f)|_U$ is closed, and since it induces a bijection between $U$ and $V$, it is a homeomorphism.
\end{proof}

\begin{oss}\label{projective space}
    In the case of the projective space $\P^n$ of dimension $n$ over $\F_q$, the set $\P^n(L)$ is in bijection with $(L^{n+1}\setminus\{0\})/L^\times$; since the latter has a natural topology induced by $L$, the former also does, and it's easy to check that it's the same as the compact topology we defined.
\end{oss}

The following statements show that the functor $F_L$ sends group schemes to topological groups.

\begin{lemma}\label{continuous functor}
The topological spaces $F_L(Y\times_{\O_L} Y)$ and $F_L(Y)\times F_L(Y)$ are naturally isomorphic.
\end{lemma}
\begin{proof}
The projections $\pi_1,\pi_2:Y\times Y\to Y$ induce a natural continuous map from $F_L(Y\times_{\O_L} Y)$ to $F_L(Y)\times F_L(Y)$. Since both spaces are compact and Hausdorff, the map is closed; since the underlying function is the natural bijection $(Y\times_{\O_L} Y)(L)\cong Y(L)\times Y(L)$, the map is a homeomorphism.
\end{proof}

\begin{prop}\label{Frobenius group scheme}
If $Y$ is a (commutative) group scheme over $\O_L$, the metric on $Y(L)$ is translation invariant, and makes it into a (commutative) topological group.
\end{prop}
\begin{proof}
By Lemma \ref{continuous functor}, we identify $F_L(Y\times_{\O_L}Y)\cong F_L(Y)\times F_L(Y)$ via a natural homeomorphism.

Call $e$ the identity, $i$ the inverse, and $m$ the multiplication of $Y$. Then $F_L(Y)$ has a natural structure of topological group, with identity $F_L(e)$, inverse $F_L(i)$ and multiplication $F_L(m)$, because all the necessary diagrams commute by functoriality. For the same reason, if $Y$ is commutative, $Y(L)$ is also commutative.

To prove the invariance of the metric, we need to show that every translation is an isometry. Fix a morphism of $\O_L$-schemes $P:\Spec(\O_L)\to Y$ (i.e. $P\in Y(\O_L)$), and consider the following:
\[l_P:Y\cong \Spec(\O_L)\times_{\O_L}Y\xrightarrow{P\times id_Y} Y\times_{\O_L}Y\xrightarrow{m}Y,\]
so that $F_L(l_P):Y(L)\to Y(L)$ is the left translation by $P$. It's immediate to check that, if we call $-P$ the inverse of $P$ in $Y(\O_L)$, $l_{-P}$ is the two-sided inverse of $l_P$, therefore they are isomorphisms.
In particular $l_P$ induces a family of bijections $\{Y(\O_L/\m_L^k)\to Y(\O_L/\m_L^k)\}_{k\geq1}$, whose limit is precisely $F_L(l_P)$, hence $F_L(l_P)$ is an isometry. The proof for right translations is essentially the same.
\end{proof}

\begin{cor}\label{series convergence} Suppose that $Y$ is commutative. Denote by addition the group law on $Y(L)$ and by $0$ its identity element. If $(P_i)_{i\in \mathbb{N}}$ is a sequence in $Y(L)$ converging to $0$, then the series $\sum_i P_i$ is a well defined element of $Y(L)$ (i.e. the sequence of partial sums converge).
\end{cor}
\begin{proof}
Call $\bar{d}$ the distance on $Y(L)$. Since $\bar{d}$ is ultrametric, we just need the limit of the distances $\bar{d}(S_k,S_{k-1})$ to be $0$, where $S_k:=\sum_{i=0}^k P_i$. Since the metric is translation invariant, $\lim_k \bar{d}(S_k,S_{k-1})=\lim_k \bar{d}(P_k,0)$, which is zero by hypothesis.
\end{proof}

\subsection{The compact topology on the space of divisors}

In this subsection we state some propositions about the symmetric powers of a curve and its Jacobian. Most results are already stated and proven in \cite{Milne}.

Recall the definition of $X$; $S_d$ is the permutation group of $d$ elements. We have the following (see \cite{Milne}[Prop. 3.1, Prop. 3.2]).

\begin{prop}\label{symmetric power}
Fix a positive integer $d$. Consider the natural right action of $S_d$ on $X^d$ and call its quotient $X^{[d]}$. Then $X^{[d]}$ is a proper smooth $\F_q$-scheme.
\end{prop}

The following result (see \cite{Milne}[Thm. 3.13]) gives us the functorial interpretation of the symmetric power $X^{[d]}$.

\begin{teo}
Consider the functor $\Div^d_X$ which sends an $\F_q$-algebra $R$ to the set of relative effective Cartier divisors of degree $d$ on $X_R$ over $R$ (i.e. effective Cartier divisors on $X_R$ which are finite and flat of rank $d$ over $R$). This functor is represented by $X^{[d]}$.
\end{teo}

\begin{cor}\label{cor Milne}
For every field extension $E/\F_q$, $X^{[d]}(E)$ is in bijection with the $E$-subschemes of $X_E$ of degree $d$.
\end{cor}

Let's continue with the fundamental property of the Jacobian variety (see \cite{Milne}[Thm. 1.1]).

\begin{teo}
Call $P^0_X$ the natural functor from $\F_q$-algebras to abelian groups such that for any $\F_q$-algebra $R$:
\[P^0_X(R)=\{\L\in\Pic(X_R)|\deg(\L_t)=0\;\forall t\in \Spec(R)\}/\pi_R^*(\Pic(R)).\]
There is an abelian variety $\A$ over $\F_q$, called the Jacobian variety of $X$, and a natural transformation of functors $P^0_X\to\A$ which induces an isomorphism $P^0_X(R)\cong\A(R)$ whenever $X(R)\neq\emptyset$.
\end{teo}
Let's fix a point $\infty'\in X(\F_{q^e})$ with support at $\infty$.
The following result clarifies the relation between the symmetric powers of $X$ and $\A$ (see \cite{Milne}[Thm. 5.2]).
\begin{teo}\label{J^d}
For all $d\geq1$, the point $\infty'\in X_{\F_{q^e}}$ induces a natural morphism of $\F_{q^e}$-schemes $J^d:X_{\F_{q^e}}^{[d]}\to\A_{\F_{q^e}}$. Moreover, the morphism $J^g:X_{\F_{q^e}}^{[g]}\to\A_{\F_{q^e}}$ is birational and surjective.
\end{teo}
\begin{oss}
For every field $E/\F_{q^e}$, at the level of $E$-points the morphism $J^d$ sends an effective divisor $D$ of degree $d$ to the class of $D-d\infty'$.
\end{oss}

Finally, a result on the fibers of the map $J^d$ (see \cite{Milne}[Rmk. 5.6.(c)] and \cite{Hartshorne}[Prop. II.7.12]).

\begin{prop}\label{fibers}
Fix a field extension $E/\F_{q^e}$ and a point $D\in X_{\F_{q^e}}^{[d]}(E)$, define $P:=J^d\circ D\in\A_{\F_{q^e}}(E)$, and call $V$ the $E$-vector space $H^0(X_E,D)$. The fiber $(J^d)^*P$ is naturally isomorphic as an $E$-scheme to $\P(V)$.

For any field extension $E'/E$, for all $f\in E'\otimes_E V\cong H^0(X_{E'},D)$, the isomorphism sends the line $E'\cdot f\in\P(V)(E')$ to $\Div(f)+D\in X^{[d]}(E')$.
\end{prop}

\begin{cor}\label{preimage of h^0=1}
Let $D\in X_{\F_{q^e}}^{[d]}(\closure{E})$ with $h^0(D)=1$. If $J^d\circ D\in\A_{\F_{q^e}}(\closure{E})$ factors through some $P\in\A_{\F_{q^e}}(E)$, $D$ factors through some $D'\in X_{\F_{q^e}}^{[d]}(E)$.
\end{cor}

\begin{proof}
Since $D$ factors through some finite extension $\Phi:\Spec(E')\to\Spec(E)$, we can assume $D\in X_{\F_{q^e}}^{[d]}(E')$ without loss of generality. By Proposition \ref{fibers}, the pullback of $P\circ\Phi\in\A_{\F_{q^e}}(E')$ along $J^d$ is a morphism $\Spec(E')\to X_{\F_{q^e}}^{[d]}$, hence it is exactly $D$. If $Z\to X_{\F_{q^e}}^{[d]}$ is the pullback of $P$ along $J^d$, $Z\times_{\Spec(E)}\Spec(E')$ is isomorphic to $\Spec(E')$; we deduce that $Z\cong\Spec(E)$, and $D$ factors through $Z\to X_{\F_{q^e}}^{[d]}$.
\end{proof}

With the following proposition we can finally switch between convergence of functions and convergence of divisors, an essential step to prove the functional identities in Theorem \ref{functional identity weak} and Theorem \ref{Sf module weak}.

\begin{prop}\label{convergence of functions and divisors}
Fix a finite field extension $L/\K$ and an effective divisor $D_-$ in $X^{[d]}(\F_{q^e})$, and consider a sequence $(h_m)_m$ in $H^0(X_L,D_-)$.

If the sequence $(\Div(h_m)+D_-)_m$ converges to $D_+\in X^{[d]}(L)$, there are $(\lambda_m)_m$ in $L^\times$ such that $(\lambda_m h_m)_m$ converges in $H^0(X_L,D_-)$ to some nonzero $h$ with $\Div(h)=D_+-D_-$.

If the sequence $(h_m)_m$ converges in $H^0(X_L,D_-)$ to some nonzero $h$, the sequence $(\Div(h_m)+D_-)_m$ converges to $\Div(h)+D_-\in X^{[d]}(L)$.
\end{prop}
\begin{proof}
Call $V:=H^0(X_{\F_{q^e}},D_-)$ and call $Z_d$ the pullback of the closed subscheme $[D_--d\infty']\in\A_{\F_{q^e}}$ along $J^d:X_{\F_{q^e}}^{[d]}\to\A_{\F_{q^e}}$, so that $\Div(h_m)+D_-\in Z_d(L)$ for all $m$. As we noted in Remark \ref{projective space}, $\P(V)(L)$ is homeomorphic to $(H^0(X_L,D_-)\setminus\{0\})/L^\times$ endowed with the quotient topology. On the other hand, by Proposition \ref{fibers} (setting $E=\F_{q^e}$ and $D=D_-$), the $\F_{q^e}$-schemes $\P(V)$ and $Z_d$ are isomorphic; in particular, the induced map $\P(V)(L)\to Z_d(L)$, which sends a line $L\cdot f\in H^0(X_L,D_-)$ to $\Div(f)+D_-$, is a homeomorphism in the compact topology, by Remark \ref{functor for F_p-schemes}.

If the sequence $(\Div(h_m)+D_-)_m$ converges to $D_+\in Z_d(L)$, this proves that the equivalence classes $([h_m])_m$ in $H^0(X_L,D_-)/L^\times$ do converge to an equivalence class $[h]$ whose divisor is $D_+-D_-$. Since the projection is open, we can lift this convergence to $H^0(X_L,D_-)$ up to scalar multiplication.

The map $H^0(X_L,D_-)\setminus\{0\}\to Z(L)$ sending a function $f$ to the effective divisor $\Div(f)+D_-$ is continuous. In particular, if the sequence $(h_m)_m$ converges to a nonzero $h\in H^0(X_L,D_-)$, the sequence $(\Div(h_m)+D_-)_m$ converges to $\Div(h)+D_-\in Z_d(L)$.
\end{proof}

\section{Frobenius and divisors}\label{Frobenius and divisors}

Fix a nonzero ideal $I<A$, with ideal class $\bar{I}\in Cl(A)$; with slight abuse of notation, call $I$ also the corresponding effective divisor of $X$. Call $\Xi\in X(K)$ the morphism $\Spec(K)\to X\setminus\{\infty\}$ corresponding to the canonical inclusion $A\into K$. 

In the first subsection, we recall the notion of Frobenius twist $P^{(1)}$ for a point $P$ in $X^{[d]}(\K)$, and study its behavior with respect to the compact topology. The main result is Proposition \ref{red limit}, where we prove that the sequence $(P^{(m)})_m$ converges to $\red_{\K}(P)$.

In the second subsection, we study the divisor of a rational function $h$ with respect to its expansion $\sum_{i\geq k}c_i u^i$ as an element of $\K\hat\otimes K\cong K((u))$. Among several useful results, the most significant is Proposition \ref{Div and red commute}, where we state the identity $\Div(c_k)=\red_\K(\Div(h))$.

Finally, in the third subsection, we construct the divisors $\{V_{\bar{I},*,m}\}$ for $m\gg0$ and $V_{\bar{I},*}$ in $X^{[g]}(\K)$ (see Proposition \ref{con}), uniquely defined by the following linear equivalences for $m\gg0$:
\begin{align*}
    &\begin{cases}V_{\bar{I},*,m}-V_{\bar{I},*,m}^{(1)}\sim\Xi^{(m)}-\Xi^{(1)}\\\red_\K(V_{\bar{I},*,m})\sim (\deg(I)+g)\infty-I\end{cases};
    &\begin{cases}V_{\bar{I},*}-V_{\bar{I},*}^{(1)}\sim\infty-\Xi\\\red_\K(V_{\bar{I},*})\sim (\deg(I)+g)\infty-I\end{cases}.
\end{align*}
The main result is the convergence of the sequence $(V_{\bar{I},*,m})_{m\gg0}$ to $V_{\bar{I},*}^{(1)}$ in $X^{[g]}(\K)$ (Proposition \ref{V_{I,m}}).

\subsection{Frobenius twist}
In this subsection we define the Frobenius twist for a (proper) $\F_q$-scheme $Y$ and study its behavior with respect to the topology of $Y(\K)$. The fundamental results are Proposition \ref{red limit} and Lemma \ref{Twist and divisors}.
\begin{Def}\label{def frobenius}
Let $Y$ be an $\F_q$-scheme, and $R$ an $\F_q$-algebra. Consider the morphism $\Frob_R:\Spec(R)\to\Spec(R)$ of $\F_q$-schemes induced by raising to the $q^\text{th}$ power. The morphism $\Frob_R$ and the identity on $Y$ induce an endomorphism of $Y_R$ over $\Spec(\F_q)$: we denote it by $F_R^Y:Y_R\to Y_R$.

Call $\pi_Y:Y_R\to Y$ and $\pi_R:Y_R\to\Spec(R)$ the natural projections. For all $P\in Y(R)$, $\overline{P}$ denotes the unique element of $\Hom_R(\Spec(R),Y_R)$ such that $P=\pi_Y\circ\overline{P}$. We call Frobenius twist of $P$, denoted by $P^{(1)}\in Y(R)$, the only element such that $\overline{P^{(1)}}$ is the pullback of $\overline{P}$ along $F_R^Y$. The $n$-th iteration of the twist is denoted by $P^{(n)}$ for all $n\in\mathbb{N}$.
\end{Def}

\begin{lemma}\label{Frob}
    In the notation of Definition \ref{def frobenius}, we have $P^{(1)}=P\circ\Frob_R$.
\end{lemma}
\begin{proof}
    We have the following cartesian diagram:
\[\begin{tikzcd}
	{\Spec(R)} & {Y_R} & {\Spec(R)} \\
	{\Spec(R)} & {Y_R} & {\Spec(R)}
	\arrow["{\pi_R}", from=1-2, to=1-3]
	\arrow["{\overline{P^{(1)}}}", from=1-1, to=1-2]
	\arrow["{\Frob_R}"', from=1-1, to=2-1]
	\arrow["{\overline{P}}"', from=2-1, to=2-2]
	\arrow["{F_R^Y}"{anchor=west}, from=1-2, to=2-2]
	\arrow["{\pi_R}"', from=2-2, to=2-3]
	\arrow["{\Frob_R}", from=1-3, to=2-3]
	\arrow["\square"{anchor=center, pos=0.5}, draw=none, from=1-2, to=2-3]
	\arrow["\square"{anchor=center, pos=0.5}, shift left=1, draw=none, from=1-1, to=2-2].
\end{tikzcd}\]
    Since $\pi_Y\circ F_R^Y=\pi_Y$, $P^{(1)}=\pi_Y\circ\overline{P^{(1)}}=\pi_Y\circ F_R^Y\circ\overline{P^{(1)}}=\pi_Y\circ\overline{P}\circ\Frob_R=P\circ\Frob_R.$
\end{proof}
\begin{oss}
    In light of Lemma \ref{Frob}, if $\Frob_R$ is an isomorphism, for all $P\in Y(R)$ we can redefine $P^{(k)}\in Y(R)$ as $P\circ(\Frob_R)^k$ for all $k\in\mathbb{Z}$.
\end{oss}

\begin{lemma}\label{Frobenius twist of a power}
Fix a positive integer $d$, an $\F_q$-scheme $Y$, and an $\F_q$-algebra $R$, and consider a point $(P_1,\dots,P_d)\in Y^d(R)$. Its Frobenius twist is $(P_1^{(1)},\dots,P_d^{(1)})$.
\end{lemma}
\begin{proof}
The $i$-th projection $\pi_i:Y^d\to Y$ is such that $\pi_i\circ(P_1,\dots,P_d)=P_i$. By Remark \ref{Frob}:
\[\pi_i\circ\left((P_1,\dots,P_d)^{(1)}\right)=\pi_i\circ(P_1,\dots,P_d)\circ\Frob_R=P_i\circ\Frob_R=P_i^{(1)}. \tag*{\qedhere}\]
\end{proof}
\begin{oss}
The analogous statement, with the same proof, is true for any product of $\F_q$-schemes.
\end{oss}

Let $L/\K$ be a finite field extension with residue field $\F_L\subseteq\O_L$ and $Y$ a proper $\O_L$-scheme. Recall the notation $\red_L$ from Definition \ref{red}.

\begin{prop}\label{red limit}
Fix a point $P\in Y(L)$, and set $k_L$ such that $\#\F_L=q^{k_L}$.
The sequence $(P^{(m k_L+r)})_m$ converges to $\red_L(P)^{(r)}$ in $Y(L)$.
\end{prop}
\begin{proof}
Since $\Spec(\O_L)$ only has one closed point, we can choose an open affine subscheme $U\subseteq Y$ with $B:=\O_Y(U)$ such that $P\in U(\O_L)$: $P$ corresponds to a map of $\O_L$-algebras $\chi_P:B\to\O_L$; its reduction modulo $\m_L$, composed with the immersion $\F_L\into\O_L$, is equal to the morphism $\chi_{\red_L(P)}:B\to\O_L$ corresponding to $\red_L(P)$ by Definition \ref{red}. For all $i$, $P^{(i)}$ corresponds to the map $(\cdot)^{q^i}\circ\chi_P$, which modulo $\m_L^{q^i}$ is the same as $\chi_{\red_L(P)^{(i)}}$, hence the projections of $P^{(i)}$ and $\red_L(P)^{(i)}$ onto $Y(\O_L/\m_L^{q^i})$ coincide. Since $\red_L(P)^{(m k_L+i)}=\red_L(P)^{(i)}$ for all $m\geq0$, this proves the convergence.
\end{proof}

\begin{oss}\label{oss twist div}
For any effective divisor $D$ of the curve $X_L$ over $\Spec(L)$, we can define its twist $D^{(1)}$ as the pullback along $F_L^X$. Obviously, if $D=\sum P_i$ with $P_i\in X_L(L_i)$, $D^{(1)}=\sum P_i^{(1)}$.
\end{oss}

\begin{Def}\label{def twist fun}
Let $h$ be a non-constant rational function on $X_L$, i.e. a non-constant morphism of $L$-schemes $X_L\to\P^1_L$. We define the Frobenius twist $h^{(1)}:=h\circ F_L^X$.
\end{Def}
\begin{oss}\label{oss twist fun}
    The field of rational functions of $X_L$ is $\Q(L\otimes A)$, and if $h=\sum_i l_i\otimes a_i$ in $L\otimes A$ is non-constant, $h^{(1)}=\sum_i l_i^q\otimes a_i$. In particular, we can naturally extend the Frobenius twist to the constant rational functions $L\otimes\F_q\subseteq L\otimes A$ as the elevation to the $q$-th power.
\end{oss}

We show that the Frobenius twists of divisors and rational functions are compatible.

\begin{lemma}\label{Twist and divisors}
Let $h\in L(X)^\times$, and call $\Div(h)$ its divisor. Then, $\Div(h^{(1)})$ is equal to $(\Div(h))^{(1)}$.
\end{lemma}
\begin{proof}
If $h$ is a constant function, both sides are the empty divisor. If $h$ is non-constant, for any closed point $P\in\P^1_L$, $(h^{(1)})^*(P)=(F_L^X)^*\circ h^*(P)$; setting $P=[0:1]$ and $P=[1:0]$, since the Frobenius twist on the divisors is induced by the pullback via $F_L^X$, we get our thesis.
\end{proof}

\subsection{Rational functions on \texorpdfstring{$X_\K$}{X} as Laurent series}
Let $L/\K$ be a finite field extension with a uniformizer $u\in\O_L$, call $\F_L\subseteq\O_L$ the residue field of $L$, and $k_L$ the integer such that $\#\F_L=q^{k_L}$. Call $K':=\F_L\otimes K$, i.e. the fraction field of $X_{\F_L}$, and $A':=\F_L\otimes A$.
\begin{oss}
   Since $L(X)$ is the fraction field of $\O_L\otimes K$, which has as a maximal ideal $\m_L\otimes K=(u\otimes1)\O_L\otimes K$, we can endow $L(X)$ with the $\m_L\otimes K$-adic metric. 
\end{oss}

\begin{lemma}\label{K((u))}
There is a natural immersion of fields $L(X)\into K'((u))$, which is a completion with respect to the $\m_L\otimes K$-adic metric. Moreover, it induces an isomorphism between the completion $L\hat\otimes A$ of $L\otimes A\subseteq L(X)$ and $A'[[u]][u^{-1}]$.
\end{lemma}

\begin{proof}
The natural isomorphisms $(\O_L\otimes K/\m_L^k\otimes K\cong K'[u]/u^n)_{n\geq1}$ pass to the limit and to fraction fields, giving a natural isometry between the completion of $L(X)$ and $K'((u))$.

The inclusion $L\otimes A\subseteq A'[[u]][u^{-1}]$ is obvious, and by the previous reasoning it is an isometry with respect to the natural metric of $L\otimes A$; on the other hand each element in $A'[[u]][u^{-1}]$ is the limit of its truncated expansions, which are in $L\otimes A$.
\end{proof}
\begin{lemma}\label{lifting convergence}
For all positive integers $d$, the induced inclusion of $H^0(X_L,d\infty)$, endowed with its natural metric of finite $L$-vector space, into $K'((u))$ is a closed immersion.
\end{lemma}
\begin{proof}
The restriction of the $\m_L\otimes K$-adic metric of $L(X)$ to $H^0(X_L,d\infty)$, which is isomorphic to $ L\otimes A(\leq d)$, is the natural metric of a finite $L$-vector space. By Lemma \ref{K((u))}, the inclusion into $K'((u))$ is an isometry, hence a closed immersion.
\end{proof}
\begin{oss}\label{twist in K((u))}
For all $h\in L(X)\subseteq K'((u))$, if we write $h=\sum_{j\geq m}c_j u^j$, with $c_j\in K'$ for all $j$, then $h^{(1)}=\sum_{j\geq m}c_j^{(1)}u^{qj}$.
\end{oss}

To better understand the usefulness of $K'((u))$, let's state a couple of propositions. First, we prove a very natural result, analogous to Lemma \ref{Twist and divisors} but with the reduction instead of the twist.
\begin{Def}
    For all nonzero $h\in L(X)\subseteq K'((u))$, write $h=\sum_{j\geq m}c_j u^j$ with $c_j\in K'$ for all $j$ and $c_m\neq0$, and set $\red_u(h):=c_m$.
\end{Def}
\begin{prop}\label{Div and red commute}
For all nonzero $h\in L(X)$, $\Div(\red_u(h))=\red_L(\Div(h))$, where both are $\F_L$-rational divisors of $X_{\F_L}$.
\end{prop}
\begin{proof}
Since for any nonzero $h\in L(X)$ there is a positive integer $d$ and $h_+,h_-$ in $\O_L\otimes_{\F_L} A'(\leq d)$ such that $h=\frac{h_+}{h_-}$, we can assume $h\in\O_L\otimes_{\F_L} A'(\leq d)$. Up to a factor in $L^\times$, we can also assume $h=\sum_{i\geq0}c_i u^i\in K'[[u]]$ with $c_0\in A'(\leq d)\setminus\{0\}$. By Remark \ref{twist in K((u))}, the sequence $(h^{(m k_L)})_m$ is equal to $(\sum_{j\geq0}c_j u^{j q^{m k_L}})_m$, hence it converges to $c_0$ in $K'[[u]]$; by Lemma \ref{lifting convergence} this convergence lifts to $L\otimes_{\F_L} A'(\leq d)$. The sequence of divisors $(\Div(h^{(m f_L)})+d\infty)_m$, by Proposition \ref{convergence of functions and divisors}, converges to $\Div(c_0)+d\infty$ in $X^{[d]}(L)$; on the other hand, by Proposition \ref{red limit}, it converges to $\red_L(\Div(h))+d\infty$, hence we have the desired equality.
\end{proof}
We prove now that the immersion $L(X)\into K'((u))$ behaves reasonably well with evaluations.
\begin{prop}\label{conv eval entire rational}
Fix $h\in H^0(X_L,d\infty)$, and expand $h=\sum_i h_{(i)} u^i$ as an element of $K'((u))$; fix $P\in X_{\F_L}(\closure{L})\setminus\{\infty\}$, corresponding to a $\F_L$-linear homomorphism $\chi_P:A'\to \closure{L}$. Then, $h_{(i)}\in A'$ for all $i$ and $h(P)=\sum_i \chi_P(h_{(i)})u^i$.
\end{prop}
\begin{proof}
We can write $h=\sum_j \gamma_j a_j$, with $\gamma_j\in L=\F_L((u))$ and $a_j\in A'(\leq d)$, hence $h_{(i)}\in A'(\leq d)$ for all $i$. For all integers $m$ define $\gamma_{j,m}$ as the truncation of $\gamma_j\in\F_L((u))$ at the degree $m$, and define $h_m:=\sum_j \gamma_{j,m} a_j\in K'[u^{\pm1}]$, so that $h_m=\sum_{i\leq m} h_{(i)} u^i$. We have the equalities:
\begin{align*}
    &h_m(P)=\chi_P\left(\sum_j \gamma_{j,m} a_j\right)=\sum_j\gamma_{j,m}\chi_P(a_j);\\
    &h_m(P)=\chi_P\left(\sum_{i\leq m} h_{(i)} u^i\right)=\sum_{i\leq m}\chi_P(h_{(i)})u^i;
\end{align*}
where we used that both summations are finite.
Since the sequence $(\gamma_{j,m})_m$ converges to $\gamma_j$ in $\F_L((u))$ for all $j$, the first equation tells us that the sequence $(h_m(P))_m$ converges to $h(P)$. From the second equation we deduce that the series $\sum_i \chi_P(h_{(i)})u^i$  also converges, and is equal to $h(P)$. 
\end{proof}

\begin{prop}\label{conv eval Tate}
Let $h=\sum_i h_{(i)} u^i\in A'[[u]][u^{-1}]$ be a rational function on $X_L$, and fix $P\in X_{\F_L}(L)$ such that $\red_L(P)\neq\infty$, corresponding to a $\F_L$-linear homomorphism $\chi_P:A'\to\O_L$. Then $P$ is not a pole of $h$, and $h(P)=\sum_i \chi_P(h_{(i)})u^i$.
\end{prop}
\begin{proof}
For $N\gg0$, the space $H^0(X_L, N\infty-\Div_-(h)-P)$ is strictly included in $H^0(X_L, N\infty-\Div_-(h))$; we can fix $h_-$ in their difference, and set $h_+:=h h_-$; by definition, $h_+,h_-\in L\otimes A$.

If we write $h_+=\sum_i h_{+,(i)} u^i$ and $h_-=\sum_i h_{-,(i)} u^i$, we have for all integers $k$ the equation $h_{+,(k)}=\sum_{i+j=k}h_{(i)} h_{-,(j)}$, which commutes with evaluation, being a finite sum. Since $\chi_P$ has image in $\O_L$, the series $\sum_i \chi_P(h_{(i)})u^i$ converges, hence by Proposition \ref{conv eval entire rational} we get the following equation in $\O_L$:
\begin{align*}
     h_+(P)&=\sum_k \chi_P(h_{+,(k)}) u^k=\sum_k\sum_{i+j=k}\chi_P(h_{(i)})\chi_P(h_{-,(j)})u^k\\
     &=\left(\sum_i\chi_P(h_{(i)})u^i\right)\left(\sum_j \chi_P(h_{-,(j)})u^j\right)=\left(\sum_i\chi_P(h_{(i)})u^i\right)h_-(P).
\end{align*}
Since $h_-\in H^0(X_L, N\infty-\Div_-(h))\setminus H^0(X_L, N\infty-\Div_-(h)-P)$, if $P$ is a pole of $h$, then $P$ is a zero of $h_-$ of the same order, and $h_+(P)\neq0$, hence we reach a contradiction by the previous equation. Since $P$ is not a pole of $h$, then $h_-(P)\neq0$, and since $h_+(P)=h(P)h_-(P)$ we get:
\[h(P)=\frac{h_+(P)}{h_-(P)}=\sum_i\chi_P(h_{(i)})u^i. \tag*{\qedhere}\]
\end{proof}
\begin{prop}\label{A[[u]][u^{-1}]}
Let $h=\sum_i h_{(i)}u^i\in K'((u))$ be a rational function on $X_L$. Then, $h$ is in $A'[[u]][u^{-1}]$ if and only if all its poles reduce to $\infty$.
\end{prop}
\begin{proof}
Suppose $h\in A'[[u]][u^{-1}]$ and take a pole $P\in X_L(E)$ of $h$, where $E/L$ is a finite field extension. Define $A'':=\F_E\otimes A$ and $K'':=\F_E\otimes K$, where $\F_E$ is the residue field of $E$, and fix a uniformizer $v$ of $\O_E$. The natural immersion $K'((u))\subseteq K''((v))$ sends $h$ into $A''[[v]][v^{-1}]$; applying Proposition \ref{conv eval Tate} to the function $h$, defined over the field $E$, $\red_E(P)=\infty$. 

If vice versa $h\not\in A'[[u]][u^{-1}]$, call $m$ the least integer such that $h_{(m)}\not\in A'$ and set $h':=\sum_{i<m} h_{(i)} u^i$.  By Lemma \ref{Div and red commute}:
\[\Div(h_{(m)})=\Div(\red_u(h-h'))=\red_L(\Div(h-h')),\]
therefore, since $h_{(m)}\not\in A'$, $h-h'$ has a pole at a point $P$ which does not reduce to $\infty$; on the other hand, since $h'\in A'[[u]][u^{-1}]$, $h'$ does not have a pole at $P$, hence $P$ is a pole of $h=(h-h')+h'$.
\end{proof}

\begin{cor}\label{cor conv eval Tate}
Let $h=\sum_i h_{(i)}u^i\in A'[[u]][u^{-1}]$ be a nonzero rational function on $X_L$, and suppose that the coefficients $(h_{(i)})_i$ are all contained in some maximal ideal $P< A'$. Then $P$, as a closed point of $X_L$, is a zero of $h$.
\end{cor}
\begin{proof}
Take a point $Q\in X_L(\closure{L})$ with support at $P$. By Proposition \ref{conv eval Tate}, we get the identity $h(Q)=\sum_i h_{(i)}(Q)u^i=0$, hence $P$ is a zero of $h$.
\end{proof}

\subsection{Notable divisors and convergence results}
As foreshadowed by Lemma \ref{Twist and divisors}, in this subsection we explore the relation between Frobenius twists, divisors, and the compact topology.

\begin{lemma}[Drinfeld's vanishing lemma]\label{Drinfeld}
Let $E/K$ be a field extension, $W$ a point in $X^{[d]}(E)$ for some $d\leq g$, $P,Q\in X(E)$. Suppose that $[W-W^{(m)}]=[P-Q]$, where $P\neq Q^{(sm)}$ for $0\leq s+d\leq2g$; then $d=g$ and $h^0(W)=1$.
\end{lemma}
\begin{proof}
Call $W_0:=W$ and set $W_{i+1}=W_i+Q^{(im)}$ for all $i\in\mathbb{Z}$. Note that, since $\deg(W_k)=d+k$, $h^0(W_{-d-1})=0$ and $h^0(W_{2g-d-1})=g$. For all $i$, we have the inequalities $h^0(W_i)\leq h^0(W_{i+1})\leq h^0(W_i)+1$, so there is a least integer $k\in[-d,0]$ such that $h^0(W_k)=1$. 

Let's prove that for all $i\in[-d,2g-d-1[$, if $h^0(W_i)\geq1$, then $h^0(W_{i+1})=h^0(W_i)+1$. We have two relations:
\begin{align*}
    W_{i+1}&=(W+Q+\dots+Q^{((i-1)m)})+Q^{(im)}=W_i+Q^{(im)},\\
    W_{i+1}&=(W+Q^{(m)}+\dots+Q^{(im)})+Q=(W_i^{(m)}-W^{(m)}+W)+Q\sim W_i^{(m)}+P;
\end{align*}
they imply that $H^0(X_E,W_i^{(m)})\subseteq H^0(X_E,W_{i+1})$ and $H^0(X_E,W_i)\subseteq H^0(X_E,W_{i+1})$. To prove that those inclusions are strict, we need that $H^0(X_E,W_i)\neq H^0(X_E,W_i^{(m)})$ as subspaces of $H^0(X_E,W_{i+1})$; they have the same dimension because the $m$-th Frobenius twist induces an isomorphism between the two vector spaces, so we just need $W_i\not\sim W_i^{(m)}$, but:
\[W_i\not\sim W_i^{(m)}\Leftrightarrow W-W^{(m)}\not\sim Q^{(im)}-Q\Leftrightarrow P\not\sim Q^{(im)}\Leftrightarrow P\neq Q^{(im)},\]
which is implied by our hypothesis. In particular, since $W_{2g-d-1}$ has degree $2g-1$, we get that
\[g=h^0(W_{2g-d-1})=h^0(W_k)+2g-d-1-k=2g-d-k,\]
therefore $g=d+k$; but $k\leq0$ and $d\leq g$ implies $d=g$ and $k=0$, therefore $h^0(W)=1$. 
\end{proof}

The previous lemma ensures that if such a divisor $W$ exists, it has no other effective divisors in its same equivalence class. On the other hand, the existence of such $W$ in some particular cases is ensured by the following results. 

As usual, let $L$ be a finite field extension of $\K$, with residue field $\F_L\subseteq\O_L$ and $q^{k_L}:=\#\F_L$

\begin{lemma}\label{finite covering}
Call $\A_0(L)$ the kernel of $\red_L:\A(L)\to\A(\F_L)$ (which is a continuous homomorphism). The map $\A(L)\to\A_0(L)\times\A(\F_L)$ sending a point $D$ to the couple $(D-D^{(k_L)},\red_L(D))$ is an isomorphism of topological groups.
\end{lemma}
\begin{proof}
The map is obviously a continuous group homomorphism. Since domain and codomain are both compact and Hausdorff, it's sufficient to prove bijectivity.

On one hand, to prove injectivity, if we suppose $D-D^{(k_L)}=0$ we have $D\in\A(\F_L)$, so if $\red_L(D)=0$ we can deduce that $D=0$.

On the other hand, to prove surjectivity, we fix $(D_0,\tilde{D})\in\A_0(L)\times\A(\F_L)$ and show that they are the image of some $D\in\A(L)$. By Proposition \ref{red limit} we have that the sequence $(D_0^{(i k_L)})_i$ converges to $\red_L(D_0)=0$, hence by Corollary \ref{series convergence} the series $\tilde{D}+\sum_{i\geq0}D_0^{(i k_L)}$ converges to some point $D\in\A(L)$. Since the Frobenius twist and the reduction $\red_L$ are continuous endomorphisms of $\A(L)$, we get the following equations:
\begin{align*}
    &D-D^{(k_L)}=\tilde{D}+\sum_{i\geq0}D_0^{(i k_L)}-\tilde{D}^{(k_L)}-\sum_{i\geq1}D_0^{(i k_L)}=D_0;&\red_L(D)=\red_L(\tilde{D})=\tilde{D},
\end{align*}
hence the image of $D$ is $(D_0,\tilde{D})$.
\end{proof}

From now on, given an effective divisor $W\in X^{[d]}(\closure{\K})$, we denote by $J(W)$ its image via the morphism $J^d:X^{[d]}\to\A$, i.e. the equivalence class $[W-d\infty]$ in the Jacobian.
\begin{prop}\label{V}
Fix a point $D\in\A(\F_q)$, and let $P,Q\in X(\K)$ such that $\red_{\K}(P)$ is equal to $\red_{\K}(Q)$, with $P\neq Q^{(s)}$ for $|s|<2g$. 

Then, there is a unique effective divisor $W$ such that: $[W-W^{(1)}]=[P-Q]$, the reduction of $J(W)$ is $D$, and $\deg(W)\leq g$.
Moreover, a fortiori, $W\in X^{[g]}(\K)$ and, if $R$ is a point in the support of $W$, $R\not\in X(\overline{\F_q})$.
\end{prop}
\begin{proof}
By Lemma \ref{finite covering} there is an element $D'\in\A(\K)$ such that $D'-D'^{(1)}=[P-Q]$ and $\red_{\K}(D')=D$. Since the morphism $J^g$ is surjective, there is a divisor $W\in X^{[g]}(\closure{\K})$ such that $J(W)=D'$. By Drinfeld's vanishing lemma, there is only one divisor of degree $\leq g$ with the requested properties, hence $h^0(W)=1$; by Corollary \ref{preimage of h^0=1}, $W$ is $\K$-rational.

Now, call $W'\leq W$ the maximal $\overline{\F_q}$-rational effective divisor ($W'\in X^{[d]}(\overline{\F_q})$), and call $G$ the group of $\K$-linear field automorphisms of $\closure{\K}$, which acts naturally on $X(\closure{\K})$. Since $W\in X^{[g]}(\K)$, it is fixed by the induced action of $G$; moreover, this action sends $X(\overline{\F_q})$ to itself, hence $W'\leq W$ is also fixed by $G$: since $W'$ is both $\K$-rational and $\overline{\F_q}$-rational, $W'\in X^{[d]}(\F_q)$. We have:
\[(W-W')-(W-W')^{(1)}=(W-W^{(1)})+(W'-W'^{(1)})=W-W^{(1)}\sim P-Q,\]
but $\deg(W-W')=g$ from Drinfeld's vanishing lemma, hence $d=\deg(W')=0$.
\end{proof}

Recall the notation of $I,\bar{I},\Xi$ from the start of this section.

\begin{lemma}\label{red(Xi)}
We have the identity $\red_\K(\Xi)=\infty$ in $X(\K)$.
\end{lemma}
\begin{proof}
Since the image of the canonical inclusion $A\into\K$ is not contained in $\O_\K$, the morphism $\Xi:\Spec(\K)\to X\setminus\{\infty\}$ does not factor through $\Spec(\O_\K)$, which means that $\red_\K(\Xi)\not\in X(\F_q)\setminus\{\infty\}$, so $\red_\K(\Xi)=\infty$.
\end{proof}

Next, we construct some notable divisors.

\begin{prop}\label{con}
The following effective divisors of $X_\K$ exist and are unique:
\begin{itemize}
    \item a divisor $V_{\bar{I}}$ of degree $\leq g$, such that $V_{\bar{I}}-V_{\bar{I}}^{(1)}\sim\Xi-\infty$ and $\red_\K(J(V_{\bar{I}}))=J(I)$;
    \item for $m\geq1$, a divisor $V_{\bar{I},m}$ of degree $\leq g$, such that $V_{\bar{I},m}-V_{\bar{I},m}^{(1)}\sim\Xi^{(1)}-\Xi^{(m+1)}$ and $\red_\K(J(V_{\bar{I}}))=J(I)$;
    \item a divisor $V_{\bar{I},*}$ of degree $\leq g$ such that $J(V_{\bar{I},*})+J(V_{\bar{I}})=0$;
    \item for $m\gg0$, a divisor $V_{\bar{I},*,m}$ of degree $\leq g$ such that $J(V_{\bar{I},*,m})+J(V_{\bar{I},m})=0$.    
\end{itemize}
Moreover, they all are in $X^{[g]}(\K)$.
\end{prop}
\begin{proof}
Let's first note that the divisors, if they exist, are well defined: since for all $a,b\in A\setminus\{0\}$ $J(aI)=J(bI)$, the properties of the divisors we want to construct only depend on the ideal class $\bar{I}\in Cl(A)$ of $I$.

Since $\red_\K(\Xi)=\infty$ by Lemma \ref{red(Xi)}, we can apply Proposition \ref{V} to $P=\Xi$ and $Q=\infty$ (resp. $P=\Xi^{(1)}$ and $Q=\Xi^{(m+1)}$ for $m\gg0$), so the divisor $V_{\bar{I}}$ (resp. $V_{\bar{I},m}$) exists, is unique, and is contained in $X^{[g]}(\K)$.

Since $J^g(\closure\K):X^{[g]}(\closure\K)\to\A(\closure\K)$ is surjective, there is at least one effective divisor $V_{\bar{I},*}$ of degree at most $g$ such that $J(V_{\bar{I},*})=-J(V_{\bar{I}})$. It has the following properties:
\begin{align*}
    &[V_{\bar{I},*}-V_{\bar{I},*}^{(1)}]=[V_{\bar{I}}^{(1)}-V_{\bar{I}}]=[\infty-\Xi];
    &\red_\K(J(V_{\bar{I},*}))=-\red_\K(J(V_{\bar{I}}))=-J(I).
\end{align*}
By Proposition \ref{V} applied to $P=\infty$ and $Q=\Xi$, $V_{\bar{I},*}$ is unique, $\K$-rational, and of degree $g$.

Similarly, the existence and uniqueness of $V_{\bar{I},*,m}$ for $m\gg0$ are ensured by the following properties: $V_{\bar{I},*,m}-V_{\bar{I},*,m}^{(1)}\sim\Xi^{(m+1)}-\Xi^{(1)}$, and $\red_\K(J(V_{\bar{I},*,m}))=-J(I)$.
\end{proof}
\begin{oss}
    The classical construction of the \emph{Drinfeld divisors} $\{V_{\bar{I}}\}_{\bar{I}}$ (see \cite{Thakur}) is simpler and gives somewhat more information than the one presented in this paper. On the other hand, a topological point of view is more consistent with the rest of the work.
\end{oss}
\begin{oss}\label{Hayes}
Recall that we fixed an inclusion $H\subseteq\K$; the divisors $\{V_{\bar{I}}\}_{\bar{I}}$ are actually $H$-rational, and the natural action of $\G(H/K)$ on this set is free and transitive (see \cite{Hayes}[Prop. 3.2, Thm. 8.5]). Call $\bar{I}^\sigma\in Cl(A)$ the element such that $V_{\bar{I}^\sigma}=V_{\bar{I}}^\sigma$. Since this action commutes with morphisms of schemes, for all $\sigma\in\G(H/K)$, for all $\bar{I}\in Cl(A)$, we have that
\[[V_{\bar{I},*}^\sigma-g\infty]=[V_{\bar{I},*}-g\infty]^\sigma=[g\infty-V_{\bar{I}}]^\sigma=[g\infty-V_{\bar{I}}^\sigma]=[g\infty-V_{\bar{I}^\sigma}]=[V_{\bar{I}^\sigma,*}-g\infty];\]
hence $V_{\bar{I},*}^\sigma=V_{\bar{I}^\sigma,*}$ by Proposition \ref{con} because of uniqueness.
\end{oss}

Finally, we state the main result of this subsection, which is central to the proof of the main theorems.

\begin{prop}\label{V_{I,m}}
The sequences $(V_{\bar{I},m})_m$ and $(V_{\bar{I},*,m})_m$ converge respectively to the divisors $V_{\bar{I}}^{(1)}$ and $V_{\bar{I},*}^{(1)}$ in $X^{[g]}(\K)$.
\end{prop}

\begin{proof}
Define $U:=\{D\in X^{[g]}(\K)|h^0(D)=1\}$, so that the restriction $J^g(\K)|_U$ induces a bijection of $U$ with its image in $\A(\K)$; by definition, $U$ is the preimage of its image, hence by Lemma \ref{local homeomorphism} the restriction $J^g(\K)|_U$ is a homeomorphism.
By Proposition \ref{con}, for $m\gg0$, $h^0(V_{\bar{I},m})=h^0(V_{\bar{I}}^{(1)})=1$, so $V_{\bar{I},m},V_{\bar{I}}^{(1)}\in U$, and it suffices to prove the convergence of their images in $\A(\K)$.

If we identify $\A(\K)$ and $\A(\F_q)\times\A_0(\K)$ by Lemma \ref{finite covering}, we have:
\begin{align*}
    \lim_m V_{\bar{I},m}=&\lim_m\left(\red_\K(J^g(V_{\bar{I},m})),[V_{\bar{I},m}-V_{\bar{I},m}^{(1)}]\right)=\lim_m\left(J(I),[\Xi^{(1)}-\Xi^{(m+1)}]\right)\\
    =&\left(J(I),[\Xi^{(1)}-\infty]\right)=\left(\red_\K(J^g(V_{\bar{I}}^{(1)})),[V_{\bar{I}}^{(1)}-V_{\bar{I}}^{(2)}]\right)=V_{\bar{I}}^{(1)},
\end{align*}
where we used that $\lim_m \Xi^{(m)}=\infty$ in $X(\K)$ by Lemma \ref{red(Xi)} and Proposition \ref{red limit}.
Similarly, for the other statement, it suffices to prove that the sequence $(J(V_{\bar{I},m,*}))_m$ converges to $J(V_{\bar{I},*}^{(1)})$, which is obvious because $J(V_{\bar{I},*}^{(1)})=-J(V_{\bar{I}}^{(1)})$ and $J(V_{\bar{I},m,*})=-J(V_{\bar{I},m})$ for all $m\gg0$.
\end{proof}

\section{Pellarin zeta functions}\label{section zeta functions}
Throughout this and all the next sections we fix a uniformizer $u\in\K$ and a nonzero ideal $I< A$. As in the previous section, we also call $I$ the corresponding closed subscheme of $X\setminus\{\infty\}$, and $d$ its degree. Without loss of generality we can assume that for all $a\in A\setminus\{0\}$ the sign $\sgn(a)$ is equal to $\red_u(a(\Xi))$.

The following definition is a generalization of the zeta functions \`a la Pellarin introduced in \cite{Pellarin2011}.

\begin{Def}
    The (partial) \emph{Pellarin zeta function} relative to $I$ is defined as the series:
    \[\zeta_I:=\sum_{a\in I\setminus\{0\}}a^{-1}\otimes a\in\K\hat\otimes A.\]
\end{Def}
In this section, we first define the rational approximations $\{\zeta_{I,m}\}$ of $\zeta_I$ and compute their divisors, in analogy to what was already done by Chung, Ngo Dac and Pellarin in the case $I=A$ (see \cite{Pellarin2021}[Lemma 2.1]). Afterwards, we use Proposition \ref{convergence of functions and divisors} to prove a functional identity regarding $\zeta_I$ in the shape of an infinite product, i.e. Theorem \ref{functional identity weak}.

A stronger version of this result (Theorem \ref{functional identity}, stated in the introduction) is proven at the end of Section \ref{section duality}. 

\subsection{The approximations of \texorpdfstring{$\zeta_I$}{the Pellarin zeta} and their divisors}\label{subsection def zeta}

For $m\in\mathbb{N}$, call $j_m$ the least integer such that $\dim_{\F_q}(I(\leq j_m))=h^0(j_m\infty-I)=m+1$. We call $a_I\in I$ the nonzero element with least degree (i.e. $a_I\in I(j_0)$) and sign $1$.
\begin{oss}\label{j_m inequality}
    Since $\deg(j_m\infty-I)+1-g\leq h^0(j_m\infty-I)\leq\deg(j_m\infty-I)+1$, we get the inequality:
    \[m+d\leq j_m\leq m+g+d.\]
    Moreover, for $m\gg0$, the rightmost inequality becomes an equality.
\end{oss}

\begin{Def}
    We set for all $m\geq0$:
    \[\zeta_{I,m}:=\sum_{a\in I\setminus\{0\}(\leq j_m)}a^{-1}\otimes a\in\K\otimes A.\]
\end{Def}

\begin{oss}
The sequence $\zeta_{I,m}$ converges to $\zeta_I$ in $\K\hat\otimes A\cong A[[u]]$.
\end{oss}

\begin{prop}\label{divisor of zeta_{I,m}}
The divisor of $\zeta_{I,m}$ is $\Xi^{(1)}+\dots+\Xi^{(m)}+I+W_m-j_m\infty$ for some effective divisor $W_m$ with $h^0(W_m)=1$. Moreover, for $m\gg0$, $j_m=m+g+d$ and $W_m=V_{\bar{I},*,m}$.
\end{prop}

The following result is similar to a well known lemma (see \cite{Goss}[Lemma 8.8.1]). We prove it in this stronger form because of its use in Section \ref{section duality}.

\begin{lemma}\label{S_n,k}
    Call $S_{n,d}(x_1,\dots,x_n)\in\F_q[x_1,\dots,x_n]$ the sum of the $d$-th powers of all the homogeneous linear polynomials. Suppose that the coefficient of monomial $x_1^{d_1}\cdots x_n^{d_n}$ in the expansion of $S_{n,d}(x_1,\dots,x_n)$ is nonzero: then, for all $1\leq j\leq n$, $\sum_{i=1}^j d_i\geq q^j-1$. In particular, if $d<q^n-1$, $S_{n,d}=0$.
\end{lemma}
\begin{proof}
    The coefficient $c_{d_1,\dots,d_n}$ of the monomial $x_1^{d_1}\cdots x_n^{d_n}$ is:\[\frac{d!}{d_1!\cdots d_n!}\sum_{a_1,\dots,a_n\in\F_q}a_1^{d_1}\cdots a_n^{d_n}=\frac{d!}{d_1!\cdots d_n!}\prod_{i=1}^n\left(\sum_{a_i\in\F_q}a_i^{d_i}\right),\]where by convention we set $0^0=1$.  On one hand, if the multinomial coefficient $\frac{d!}{d_1!\cdots d_n!}$ is nonzero in $\F_q$, $C(d)=C(d_1)+\cdots+C(d_n)$, where we denote by $C(m)$ the sum of the digits in base $q$ of the nonnegative integer $m$; in particular, for $1\leq j\leq n$ this implies $C(d_1+\cdots+d_j)=C(d_1)+\cdots+C(d_j)$. On the other hand, $\sum_{a_i\in\F_q}a_i^{d_i}\neq0$ if and only if $d_i>0$ and $q-1|d_i$; in particular, this implies $C(d_i)\geq q-1$ for all $i$.
    
    If $c_{d_1,\dots,d_n}\neq0$, for $1\leq j\leq n$ we have:\[C\left(\sum_{i=1}^j d_i\right)=\sum_{i=1}^j C(d_i)\geq(q-1)j,\]hence $\sum_{i=1}^j d_i\geq q^j-1$. Applying this to $j=n$ we get the condition $d\geq q^n-1$, therefore $S_{n,d}=0$ for all $d<q^n-1$.
\end{proof}

\begin{proof}[Proof of Proposition \ref{divisor of zeta_{I,m}}]
Since $\zeta_{I,m}$ is sum of elements whose divisor contains $I$, it's obvious that $\Div^+(\zeta_{I,m})\geq I$. If we fix an $\F_q$-basis $\{a_i\}_{i=0,\dots,m}$ of $I(\leq j_m)$, for any positive integer $k$ we have:
\[\zeta_{I,m}(\Xi^{(k)})=\sum_{a\in I(\leq j_m)}a^{q^k-1}=S_{m+1,q^k-1}(a_0,\dots,a_m),\]
which by Lemma \ref{S_n,k} is zero when $k\leq m$. Since the only poles are at $\infty$, and have multiplicity at most $j_m$, $\Div(\zeta_{I,m})=\Xi^{(1)}+\dots+\Xi^{(m)}+I+W_m-j_m\infty$ for some effective divisor $W_m$. To study $h^0(W_m)$, call $D_n:=j_m\infty-I-\sum_{i=1}^n\Xi^{(i)}$ for all nonnegative integers $n$. 

Note that, since $(j_m\infty-I)^{(1)}=j_m\infty-I$, we deduce for all $n\geq0$: 
\begin{align*}
    &H^0(X_{\closure\K},D_{n+1})\subseteq H^0(X_{\closure\K},D_n),\\
    &H^0(X_{\closure\K},D_{n+1})\subseteq H^0(X_{\closure\K},D_n^{(1)}),\\
    &H^0(X_{\closure\K},D_n)\cap H^0(X_{\closure\K},D_n^{(1)})=H^0(X_{\closure\K},D_{n+1}).
\end{align*}
Let's prove that, for all $n\geq0$, if $h^0(D_n)\geq1$, then $h^0(D_{n+1})=h^0(D_n)-1$. By contradiction, assume that the set $S:=\{n\in\mathbb{Z}_{\geq0}|h^0(D_{n+1})=h^0(D_n)>0\}$ is not empty. Since for $k\gg0$ $h^0(D_k)=0$, $S$ admits a maximum element $n$; since $D_n^{(1)}>D_{n+1}$ and $D_n^{(1)}>D_{n+1}^{(1)}$, we have $H^0(X_{{\closure\K}},D_{n+1}^{(1)})+H^0(X_{{\closure\K}},D_{n+1})\subseteq H^0(X_{{\closure\K}},D_n^{(1)})$; since $h^0(D_{n+1})=h^0(D_n)=h^0(D_n^{(1)})$ we get the following identities:
\begin{align*}
&H^0(X_{\closure\K},D_n^{(1)})=H^0(X_{\closure\K},D_{n+1})=H^0(X_{\closure\K},D_{n+1}^{(1)})\\
\Rightarrow&H^0(X_{\closure\K},D_{n+1})=H^0(X_{\closure\K},D_{n+1})\cap H^0(X_{\closure\K},D_{n+1}^{(1)})=H^0(X_{\closure\K},D_{n+2})\\
\Rightarrow&h^0(D_{n+2})=h^0(D_{n+1}).
\end{align*}
We deduce $n+1\in S$, which contradicts the maximality hypothesis on $n$, therefore $S=\emptyset$. In particular $m\not\in S$, and since $h^0(W_m)\geq1$ and $W_m\sim D_m$, we have:
\[h^0(W_m)=h^0(D_m)=h^0(D_0)-m=h^0(j_m\infty-I)-m=1.\]

On one hand, $\deg(W_m)=\deg(j_m\infty-\Xi^{(1)}-\dots-\Xi^{(m)}-I)=j_m-m-d$, which is $\leq g$ by Remark \ref{j_m inequality}. On the other hand, by Lemma \ref{Twist and divisors} and Proposition \ref{Div and red commute} we have:
\begin{align*}
    &0\sim\Div(\zeta_{I,m})-\Div(\zeta_{I,m})^{(1)}=\Xi^{(1)}-\Xi^{(m+1)}+W_m-W_m^{(1)},\\
    &0\sim\Div(\red_u(\zeta_{I,m}))\sim\red_\K(\Div(\zeta_{I,m}))=I+\red_\K(W_m)-(j_m-m)\infty;
\end{align*}
so $[W_m-W_m^{(1)}]=[\Xi^{(m+1)}-\Xi^{(1)}]$, and $\red_\K(J(W_m))=-J(I)$. Therefore, for $m\gg0$, $W_m=V_{\bar{I},*,m}$ by Proposition \ref{con}.
\end{proof}

\subsection{The \texorpdfstring{function $\zeta_I$}{Pellarin zeta function} as an infinite product}

\begin{prop}\label{functions}
In $\K(X)\subseteq K((u))$ there are $f'_{\bar{I},*},f'_{\bar{I}}\in 1+uK[[u]]$, with divisors $V_{\bar{I},*}-V_{\bar{I},*}^{(1)}+\Xi-\infty$ and $V_{\bar{I}}^{(1)}-V_{\bar{I}}+\Xi-\infty$, respectively. Moreover, there is a rational function $\delta'_{\bar{I}}\in\K(X)$, with divisor $V_{\bar{I}}+V_{\bar{I},*}-2g\infty$, such that $\frac{{\delta'_{\bar{I}}}^{(1)}}{\delta'_{\bar{I}}}=\frac{f'_{\bar{I}}}{f'_{\bar{I},*}}$.
\end{prop}
\begin{proof}
From the definition of $V_{\bar{I},*}$, the divisor $V_{\bar{I},*}-V_{\bar{I},*}^{(1)}+\Xi-\infty$ is principal, hence it is the divisor of some $f'_{\bar{I},*}\in\K(X)$. Moreover, by Lemma $\ref{Div and red commute}$:
\[\Div(\red_u(f'_{\bar{I},*}))=\red_\K(\Div(f'_{\bar{I},*}))=\red_\K(V_{\bar{I},*})-\red_\K(V_{\bar{I},*})^{(1)}=0,\]
hence $\red_u(f'_{\bar{I},*})\in\F_q$; up to scalar multiplication, we can assume $f'_{\bar{I},*}=1+O(u)$. The existence of $f'_{\bar{I}}$ can be proven in the same way.

Since $V_{\bar{I}}+V_{\bar{I},*}-2g\infty$ is principal, it is the divisor of some rational function $\tilde\delta'_{\bar{I}}$ contained in $H^0(X_\K,2g\infty)\subseteq A[[u]][u^{-1}]$, and up to scalar multiplication we can assume $\tilde\delta'_{\bar{I}}=c_0+O(u)$ for some $c_0\in A$. We get:
\[\Div(\tilde\delta'_{\bar{I}})^{(1)}-\Div(\tilde\delta'_{\bar{I}})=\Div(f'_{\bar{I}})-\Div(f'_{\bar{I},*})\Longrightarrow\frac{\tilde\delta'_{\bar{I}}{}^{(1)}}{\tilde\delta'_{\bar{I}}}=\lambda\frac{f'_{\bar{I}}}{f'_{\bar{I},*}}\]
for some $\lambda\in\K$; moreover, by considering the expansion in $K((u))$, $\lambda=1+O(u)$, hence it admits a $q-1$-th root $\mu\in\O_\K$. If we set $\delta'_{\bar{I}}:=\mu^{-1}\tilde\delta'_{\bar{I}}$ we obtain the desired equation. 
\end{proof}
\begin{oss}
The choices of $f'_{\bar{I}},f'_{\bar{I},*},\delta'_{\bar{I}}$ are not unique.
\end{oss}

\begin{teo}[Weak version of Thm. \ref{functional identity}]\label{functional identity weak}
The product $-(a_I^{-1}\otimes a_I)\prod_{i\geq1}{f'_{\bar{I},*}}^{(i)}$ exists in $\O_\K\hat\otimes K$ and is equal to $\zeta_I$ up to a factor $\lambda\in\O_\K^\times$. We can also write:
\[\zeta_I=-(a_I^{-1}\otimes a_I)\prod_{i\geq0}\left((\lambda\otimes1)^{1-q}f'_{\bar{I},*}{}^{(1)}\right)^{(i)}.\]
\end{teo}
\begin{proof}
Let's identify $\O_\K\hat\otimes K$ with $K[[u]]$. By Proposition \ref{functions}, $f'_{\bar{I},*}=1+O(u)$, hence ${f'_{\bar{I},*}}^{(i)}=1+O(u^{q^i})$ for all $i\geq0$, and the convergence of the infinite product is obvious. For all $m\geq0$:
\[\red_u(\zeta_{I,m})=\red_u\left(\sum_{a\in I\setminus\{0\}}a^{-1}\otimes a\right)=\sum_{\mu\in\F_q^\times}\red_u(\mu a_I)^{-1}\otimes(\mu a_I)=-1\otimes a_I.\]
In particular, by Proposition \ref{Div and red commute}, for $m\gg0$ we have:
\[\Div(1\otimes a_I)=\Div(\red_u(\zeta_{I,m}))=\red_\K(\Div(\zeta_{I,m}))=I+\red_\K(V_{\bar{I},*,m})-(g+d)\infty;\]
since $\red_\K:X^{[g]}(\K)\to X^{[g]}(\K)$ is a continuous map, and the sequence $(V_{\bar{I},*,m})_m$ converges to $V_{\bar{I},*}$ in $X^{[g]}(\K)$ by Lemma \ref{V_{I,m}}, the equality passes to the limit:
\[\red_\K(V_{\bar{I},*})=\Div(1\otimes a_I)+(g+d)\infty-I.\]
Define the rational function $\alpha_m:={\delta'_{\bar{I}}}^{(1)}\frac{\zeta_{I,m}}{{f'_{\bar{I},*}}^{(1)}\cdots {f'_{\bar{I},*}}^{(m)}}$ for $m\gg0$ and look at its divisor:
\[\Div(\alpha_m)=I+V_{\bar{I},*}^{(m+1)}+V_{\bar{I},*,m}+V_{\bar{I}}^{(1)}-(3g+d)\infty\Longrightarrow \alpha_m\in H^0(X_\K,(3g+d)\infty).\]
By Lemma \ref{V_{I,m}}, the sequence $(\Div(\alpha_m)+(3g+d)\infty)_m$ converges to
\[I+\red_\K(V_{\bar{I},*})+V_{\bar{I},*}^{(1)}+V_{\bar{I}}^{(1)}=(\Div(1\otimes a_I)+(g+d)\infty)+(\Div({\delta'_{\bar{I}}}^{(1)})+2g\infty)\]
in $X^{[3g+d]}(\K)$. Moreover, since the sequence $(\alpha_m)_m$ converges to ${\delta'_{\bar{I}}}^{(1)}\zeta_I\left(\prod_{i\geq1}{f'_{\bar{I},*}}^{(i)}\right)^{-1}$ in $K((u))$, by Lemma \ref{lifting convergence} the latter is an element of $H^0(X_\K,(3g+d)\infty)$. By Proposition \ref{convergence of functions and divisors}, we have:
\[\Div\left({\delta'_{\bar{I}}}^{(1)}\frac{\zeta_I}{\prod_{i\geq1}{f'_{\bar{I},*}}^{(i)}}\right)=\Div(\lim_m\alpha_m)=\lim_m\Div(\alpha_m)=\Div(1\otimes a_I)+\Div(\delta_{\bar{I}}'^{(1)}).\]
In particular, there is some $\lambda\in \K$ (a fortiori in $\O_\K$) such that:
\[\zeta_I=-(\lambda\otimes1)(a_I^{-1}\otimes a_I)\prod_{i\geq1}{f'_{\bar{I},*}}^{(i)}.\]

As elements of $K((u))$, $-\zeta_I(a_I\otimes a_I^{-1})=1+O(u)$, and ${f'_{\bar{I},*}}^{(i)}=1+O(u)$ for all $i\geq0$, hence $\lambda\otimes1=1+u\F_q[[u]]\subseteq\F_q((u))$. In particular, the infinite product $\prod_{i\geq0}(\lambda^{1-q}\otimes1)^{q^i}$ converges in $\F_q[[u]]$ to $\lambda\otimes1$, so we deduce the following rearrangement:
\[\zeta_I=-(a_I^{-1}\otimes a_I)\prod_{i\geq0}\left((\lambda^{1-q}\otimes1)f'_{\bar{I},*}{}^{(1)}\right)^{(i)}. \tag*{\qedhere}\]
\end{proof}

\begin{Def}\label{def f, f_*, delta}
    Define the functions $f_{\bar{I}},f_{\bar{I},*},\delta_{\bar{I}}$ respectively as the unique scalar multiples of the functions $f'_{\bar{I}}.f'_{\bar{I},*},\delta'_{\bar{I}}$ such that $\sgn(f_{\bar{I}})=\sgn(f_{\bar{I},*})=\sgn(\delta_{\bar{I}})=1$.
    
    We call $\{f_{\bar{I}}\}_{\bar{I}\in Cl(A)}$ the \emph{shtuka functions} and $\{f_{\bar{I},*}\}_{\bar{I}\in Cl(A)}$ the \emph{adjoint shtuka functions}.
\end{Def}
\begin{oss}\label{delta}
    We have the equality $\frac{\delta_{\bar{I}}^{(1)}}{\delta_{\bar{I}}}=\frac{f_{\bar{I}}}{f_{\bar{I},*}}$, since both sides have the same divisor and the same sign.
\end{oss}
\begin{ex}
    Let's provide some insight on the adjoint shtuka functions in the low genus cases.
    
    If $g=0$, i.e. $A=\F_q[\theta]$, there is no dependence on the ideal class $\bar{I}$, and the divisors $V$ and $V_*$ have degree $0$: in particular, $\Div(f_*)=\Xi-\infty=\Div(f)$, hence $f_*=f=1\otimes\theta-\theta\otimes1$, which is in $\K\otimes A$.

    If $g=1$, i.e. $X$ is an elliptic curve, $V_{\bar{I}}$ and $V_{\bar{I},*}$ have degree $1$, and since the divisor of $\delta_{\bar{I}})$ is $V_{\bar{I}}+V_{\bar{I},*}-2\infty\sim0$, $V_{\bar{I},*}$ is the inverse of $V_{\bar{I}}$ with respect to the group operation on $X(\K)$. Suppose that we can fix an isomorphism $A\cong\F_q[x,y]/(y^2-P(x))$ with $\deg(P)=3$. If $V_{\bar{I}}$ is the map $(x,y)\mapsto (a,b)$ for some $a,b\in\K$, then $V_{\bar{I},*}$ is the map $(x,y)\mapsto (a,-b)$, and assuming $\sgn(x)=1$ we can write $\delta_{\bar{I}}=1\otimes x-a\otimes1\in\K\otimes A$, hence:
    \[f_{\bar{I},*}=f_{\bar{I}}\frac{1\otimes x-a\otimes1}{1\otimes x-a^q\otimes1}.\]
\end{ex}

\begin{oss}\label{Hayes2}
    The functions $\{f_{\bar{I}},f_{\bar{I},*},\delta_{\bar{I}}\}_{\bar{I}\in Cl(A)}$ all have sign equal to $1$, and their divisors are all $H$-rational by Remark \ref{Hayes}, so all these functions are in $\Q(H\otimes A)$. 
    From Remark \ref{Hayes} we also know that, for all $\bar{I}\in Cl(A),\sigma\in G(H/K)\cong Cl(A)$:
    \[\Div(f_{\bar{I}}^\sigma)=\Div(f_{\bar{I}})^\sigma=\left(V_{\bar{I}}^{(1)}\right)^\sigma-V_{\bar{I}}^\sigma+\Xi-\infty=V_{\bar{I}^\sigma}^{(1)}-V_{\bar{I}^\sigma}+\Xi-\infty=\Div(f_{\bar{I}^\sigma}),\]
    and since both functions have sign equal to $1$ we get $f_{\bar{I}}^\sigma=f_{\bar{I}^\sigma}$. Similarly, $f_{\bar{I},*}^\sigma=f_{\bar{I}^\sigma,*}$ and $\delta_{\bar{I}}^\sigma=\delta_{\bar{I}^\sigma}$.
\end{oss}
\begin{cor}\label{gamma_I}
There is $\gamma_I\in \C$, unique up to a factor in $\F_q^\times$, such that:
\[\frac{((\gamma_I\otimes1)\zeta_I)^{(-1)}}{(\gamma_I\otimes1)\zeta_I}=f_{\bar{I},*}.\]
\end{cor}

\section{The module of special functions}\label{section special functions}

Fix an  ideal $I< A$. By the \emph{shtuka correspondence} (see \cite{Goss}[Section 6.2]), we can associate a normalized Drinfeld module $\phi$ of rank $1$ to the shtuka function $f_{\bar{I}}$.

Let's quickly review the notion of special function, introduced in \cite{ANDTR}.

\begin{Def}[Special functions]
Set $\T:=\C\hat\otimes A$. The set of special functions relative to the Drinfeld module $\phi$ is defined as $\Sf(\phi):=\{\omega\in\T|\omega^{(1)}=f_{\bar{I}}\omega\}$.
\end{Def}
\begin{oss}
    The following fundamental property holds for all $\omega\in\T$ (see \cite{ANDTR}[Lemma 3.6] and \cite{ANDTR}[Rmk. 3.10]):
    \[\omega\in\Sf(\phi)\Longleftrightarrow\forall a\in A\;\phi_a(\omega)=(a\otimes1)\omega.\]
\end{oss}

In this section, we use Theorem \ref{functional identity weak} (in its partial version) to describe somewhat explicitly the module of special functions relative to $\phi$.

Set $\zeta:=(\gamma_I\otimes1)\zeta_I$, with $\gamma_I$ defined as in Corollary \ref{gamma_I}, so that $\zeta^{(-1)}=f_{\bar{I},*}\zeta$.
\begin{teo}[Weak version of Theorem \ref{Sf module} ]\label{Sf module weak}
The $A$-module $\Sf(\phi)$ coincides with $(\F_q\otimes I)\frac{\delta_{\bar{I}}}{\zeta^{(-1)}}$.
\end{teo}

\begin{oss}
    An interesting consequence of this result is that $\Sf(\phi)$ and $I$ are isomorphic as $A$-modules. Except for the Carlitz modules, and the Drinfeld modules studied by Green and Papanikolas in \cite{Green}, the description of the isomorphism class of the module of special functions was an open problem until Gazda and Maurischat solved it in the broad context of Anderson modules with the following Theorem (see \cite{Gazda}[Thm. 3.11]).
\end{oss}
\begin{teo*}[Gazda--Maurischat]
    Call $\Lambda$ the period lattice of $\phi$ and $\Omega$ the module of K\"ahler differentials of $A$. The $A$-module of special functions $\Sf(\phi)$ is isomorphic to $\Lambda\otimes_A\Omega^{-1}$.
\end{teo*}
\begin{cor}\label{cor Gazda}
   The period lattice $\Lambda$ of $\phi$ is isomorphic as an $A$-module to $I\otimes_A\Omega$.
\end{cor}

\begin{oss}\label{oss f_I}
    A fortiori, we can define $f_{\bar{I}}$ as the shtuka function of the unique normalized Drinfeld module of rank $1$ whose lattice is isomorphic to $I\otimes_A\Omega$.
\end{oss}

Before the proof of Theorem \ref{Sf module weak}, let's state some preliminary results.

\begin{oss}
By Lemma \ref{K((u))}, we know that $\K\hat\otimes A\cong A[[u]][u^{-1}]$. A rational function on $X_\K$ is in $\T$ if and only if it's contained in $A[[u]][u^{-1}]$, which by Proposition \ref{A[[u]][u^{-1}]} happens if and only if its poles all reduce to $\infty$.
\end{oss}

\begin{lemma}\label{fixed points of the twist}
    The subset of $\C\hat\otimes K$ fixed by the Frobenius twist is $\F_q\otimes K$.
\end{lemma}
\begin{proof}
    Fix an $\F_q$-basis $\{b_i\}_i$ of $K$: any element $c\in\C\hat\otimes K$ can be written in a unique way as a possibly infinite sum $\sum_i a_i\otimes b_i$, with $a_i\in\C$ for all $i$. If $c=c^{(1)}$, we need to have for all $i$ the equality $a_i^q=a_i$, hence $a_i\in\F_q$ for all $i$.
\end{proof}

\begin{proof}[Proof of Theorem \ref{Sf module weak}]
First, let's show that $(\F_q\otimes K)\Sf(\phi)=(\F_q\otimes K)\frac{\delta_{\bar{I}}}{\zeta^{(-1)}}$. Pick any $\omega\in \Sf(\phi)$; since $\omega^{(1)}=f_{\bar{I}}\omega$, $\delta_{\bar{I}}^{(1)}=\frac{f_{\bar{I}}}{f_{\bar{I},*}}\delta_{\bar{I}}$, and $\zeta=\frac{1}{f_{\bar{I},*}}\zeta^{(-1)}$, we have:
\[\left(\frac{\omega\zeta^{(-1)}}{\delta_{\bar{I}}}\right)^{(1)}=\frac{\omega^{(1)}\zeta}{\delta_{\bar{I}}^{(1)}}=\frac{(f_{\bar{I}}\omega)(f_{\bar{I},*}^{-1}\zeta^{(-1)})}{f_{\bar{I}}f_{\bar{I},*}^{-1}\delta_{\bar{I}}}=\frac{\omega\zeta^{(-1)}}{\delta_{\bar{I}}},\]
hence $\frac{\omega\zeta^{(-1)}}{\delta_{\bar{I}}}\in \F_q\otimes K$ by Lemma \ref{fixed points of the twist}, or equivalently $(\F_q\otimes K)\omega=(\F_q\otimes K)\frac{\delta_{\bar{I}}}{\zeta^{(-1)}}$.

We can twist everything and multiply by $\gamma_I\otimes1$ without loss of generality: the thesis is now that $(1\otimes\lambda)\frac{\delta_{\bar{I}}^{(1)}}{\zeta_I}\in A[[u]][u^{-1}]$ if and only if $\lambda\in I$. By Proposition \ref{con}, for all integers $m\gg0$ there is a function $\delta_{\bar{I},m}\in\K(X)$ with divisor $V_{\bar{I},m}+V_{\bar{I},*,m}-2g\infty$. By Proposition \ref{V_{I,m}} the sequence $(V_{\bar{I},m}+V_{\bar{I},*,m})_m$ converges to $V_{\bar{I}}^{(1)}+V_{\bar{I},*}^{(1)}\in X^{[2g]}(\K)$, hence by Proposition \ref{convergence of functions and divisors} we can choose each $\delta_{\bar{I},m}$ so that the sequence $(\delta_{\bar{I},m})_m$ converges to $\delta_{\bar{I}}^{(1)}$ in $K((u))$.

Suppose $\lambda\in I$, and consider the sequence $\left((1\otimes\lambda)\frac{\delta_{\bar{I},m}}{\zeta_{I,m}}\right)_m$ in $K((u))$, whose limit is $(1\otimes\lambda)\frac{\delta_{\bar{I}}^{(1)}}{\zeta_I}$. The divisor of the $m$-th element of the sequence (for $m\gg0$) is
\[V_{\bar{I},m}-(\Xi^{(1)}+\dots+\Xi^{(m)})-I+(m+d-g)\infty+\Div(1\otimes\lambda);\] since $\lambda\in I$, the only poles of the function reduce to $\infty$, hence $(1\otimes\lambda)\frac{\delta_{\bar{I},m}}{\zeta_{I,m}}\in A[[u]][u^{-1}]$ by Proposition \ref{A[[u]][u^{-1}]}, and so does the limit.

Vice versa, suppose $(1\otimes\lambda)\frac{\delta_{\bar{I}}^{(1)}}{\zeta_I}\in A[[u]][u^{-1}]$. Since the coefficients of $(1\otimes\lambda^{-1})\zeta_I$, as a series in $K((u))$, are all contained in $\lambda^{-1}I$, $\delta_{\bar{I}}^{(1)}=\left((1\otimes\lambda)\frac{\delta_{\bar{I}}^{(1)}}{\zeta_I}\right)\left((1\otimes\lambda^{-1})\zeta_I\right)$ has all coefficients in $\lambda^{-1}I$, so the same is true for $\delta_{\bar{I}}$. If by contradiction $\lambda\not\in I$, there is a nonzero prime ideal $P< A$ which divides the fractional ideal $\lambda^{-1}I$, hence all the coefficients of $\delta_{\bar{I}}$ are in $A\cap\lambda^{-1}I\subseteq P$, which by Corollary \ref{cor conv eval Tate} means that $P$ is a zero of $\delta_{\bar{I}}$. Since $P$, as a closed point of $X$, is $\overline{\F_q}$-rational and $\Div(\delta_{\bar{I}})=V_{\bar{I}}+V_{\bar{I},*}-2g\infty$, this is a contradiction because, by Proposition \ref{V}, neither $V_{\bar{I}}$ nor $V_{\bar{I},*}$ have $\overline{\F_q}$-rational points in their support.
\end{proof}

To end this section, let's include an analogous result to Theorem \ref{functional identity weak} for special functions.
\begin{teo}\label{omega infinite product}
There is some $\alpha\in K_\infty^\times$ such that the following element of $\C\hat\otimes K$ is well defined (up to the choice of a $q-1$-th root of $\alpha$):
\[\omega:=(\alpha\otimes1)^\frac{1}{q-1}\prod_{i\geq0}\left(\frac{\alpha\otimes1}{f_{\bar{I}}}\right)^{(i)}.\] 
Moreover, $\omega\in(\F_q\otimes K)\Sf(\phi)$.
\end{teo}
\begin{proof}
    Fix an isomorphism $\K\cong\F_q((u))$. By Proposition \ref{functions}, we can choose some nonzero $\alpha\in\F_q((u))$ such that $\alpha^{-1}f_{\bar{I}}=1+O(u)$, hence the product $\prod_{i\geq0}\left(\frac{\alpha\otimes1}{f_{\bar{I}}}\right)^{(i)}$ converges in $\K\hat\otimes K\cong K((u))$, and $\omega$ is well defined up to the choice of $\alpha^\frac{1}{q-1}$. We have:
    \[\frac{\omega^{(1)}}{\omega}=\left((\alpha\otimes1)^\frac{1}{q-1}\right)^q\prod_{i\geq0}\left(\frac{\alpha\otimes1}{f_{\bar{I}}}\right)^{(i+1)}\left((\alpha\otimes1)^\frac{1}{q-1}\prod_{i\geq0}\left(\frac{\alpha\otimes1}{f_{\bar{I}}}\right)^{(i)}\right)^{-1}=f_{\bar{I}},\]
    so $\omega\in(\F_q\otimes K)\Sf(\phi)$ by the same considerations expressed in the proof of Theorem \ref{Sf module weak}.
\end{proof}

\begin{oss}
    It's not difficult to observe that if $\beta,\gamma\in K_\infty^\times$ are such that the infinite products $\omega(\beta):=(\beta\otimes1)^\frac{1}{q-1}\prod_{i\geq0}\left(\frac{\beta\otimes1}{f_{\bar{I}}}\right)^{(i)}$ and $\omega(\gamma):=(\gamma\otimes1)^\frac{1}{q-1}\prod_{i\geq0}\left(\frac{\gamma\otimes1}{f_{\bar{I}}}\right)^{(i)}$ are well defined, $\omega(\beta)$ is equal to $\omega(\gamma)$ up to a factor in $\F_q$.
\end{oss}

\section{Relation between Pellarin zeta functions and period lattices}\label{section duality}

The aim of this section is to compute more explicitly the constant $\gamma_I$ defined in Corollary \ref{gamma_I}. To do so, we first study more in depth the zeta function $\zeta_I$ and its coefficients as a series in $K[[u]]$; afterwards, we draw a correspondence between the adjoint shtuka function $f_{\bar{I},*}$ and a certain normalized Drinfeld module $\phi$ of rank $1$, obtaining the following result.
\begin{prop*}[Prop. \ref{Lambda=I}]
    The period lattice of $\phi$ is $\gamma_I^{-1}I\subseteq\C$.
\end{prop*}

Finally, we state Theorem \ref{Sf module} (a stronger version of Theorem \ref{Sf module weak} which properly generalizes \cite{Green}[Thm. 7.1]) and Theorem \ref{functional identity} (a stronger version of Theorem \ref{functional identity weak}).

\subsection{Evaluations of the Pellarin zeta functions}

The aim of this subsection, expressed in the following proposition, is to show that there is a well behaved notion of evaluation for the Pellarin zeta function $\zeta_I$ at any point $P\in X(\C)\setminus\{\infty\}$.

From now on, for any series $s\in K[[u^\frac{1}{q^n}]][u^{-1}]$ for some $n$, we denote by $s_{(i)}$ the coefficient of $u^i$, and by $v(s)$ the least element in $\frac{1}{q^n}\mathbb{Z}$ such that $s_{(v(s))}\neq0$. By $v$ we also denote the valuation on $\C$ with the property $v(u)=1$. Finally, recall that $d$ is defined as the degree of the ideal $I$.

\begin{prop}\label{conv eval zeta}
For any point $P\in X(\C)$ different from $\infty$, which corresponds to a map $\chi_P:A\to\C$, the sequence $(\zeta_{I,m}(P))_m$ and the series $\sum_{i\geq0}\chi_P\left((\zeta_I)_{(i)}\right)u^i$ converge to the same element of $\C$.
\end{prop}
To prove the proposition, we first need some results on the coefficients $\left((\zeta_I)_{(i)}\right)_i$. 
\begin{lemma}\label{(zeta_{I,m})_{(i)}}
For all integers $i\geq0$, we have $\deg((\zeta_{I,m})_{(i)})\leq\log_q(i+1)+g+d+1$ for $m\geq0$.
\end{lemma}
\begin{proof}
Recall the definition of $j_m$, and that $m+d+1\leq j_m\leq m+g+d$, from Remark \ref{j_m inequality}. The coefficients of $\zeta_{I,0}$ have degree $j_0\leq g+d$, so the lemma holds for $m=0$. Since $v(\zeta_{I,m})=j_0$ for all $m\geq0$, the coefficient $(\zeta_{I,m})_{(0)}$, is nonzero if and only if $I=A$; in that case, it's equal to $\sum_{a\in\F_q^\times}a^{-1}\otimes a=-1$, and its valuation is $0$, so the lemma also holds for $i=0$. Let's prove the lemma for $i\geq1$, $m\geq1$.

We claim that it suffices to prove the following inequality, for all $m\geq1$:
\[v\left(\sum_{a\in I(j_m)}a^{-1}\otimes a\right)=v(\zeta_{I,m}-\zeta_{I,m-1})\geq q^{m-1}.\]
If the inequality is true for $m\geq1$, fix $i>0$, and set $n:=\lfloor\log_q(i)\rfloor+1$, so that $q^{n-1}\leq i<q^n$; then, for $m\geq n$:
\[\deg((\zeta_{I,m})_{(i)})=\deg\left(\left(\sum_{k=0}^m\zeta_{I,k}-\zeta_{I,k-1}\right)_{(i)}\right)=\deg\left(\left(\sum_{k=0}^n\zeta_{I,k}-\zeta_{I,k-1}\right)_{(i)}\right)\leq j_n,\]
which is at most $n+g+d=\lfloor\log_q(i)\rfloor+g+d+1\leq\log_q(i+1)+g+d+1$.

For $m\geq1$, $\zeta_{I,m}(\Xi)-\zeta_{I,m-1}(\Xi)=1-1=0$. By Proposition \ref{divisor of zeta_{I,m}}, on one hand, $\zeta_{I,m}-\zeta_{I,m-1}$ has only one pole, of degree at most $j_m$, at $\infty$, and has $I$ and $\Xi^{(1)},\dots,\Xi^{(m-1)}$ among its zeroes; on the other hand, 
\[h^0(W_m)=h^0(j_m\infty-I-\Xi-\cdots-\Xi^{(m-1)})=h^0(j_m\infty-I-\Xi^{(1)}-\cdots-\Xi^{(m)})=1,\]
hence the remaining set of zeroes is $W_m^{(-1)}$, and $\zeta_{I,m}-\zeta_{I,m-1}$ is a scalar multiple of $\zeta_{I,m}^{(-1)}$. 

If we fix $b\in I(j_m)$ with $\sgn(b)=1$, we get the following:
\begin{align*}
    &\left((\zeta_{I,m}-\zeta_{I,m-1})(\Xi^{(-1)})\right)^q=\sum_{a\in I(j_m)}a^{1-q}=-\sum_{\substack{a\in I(j_m)\\\sgn(a)=1}}a^{1-q}=-\sum_{c\in I(<j_m)}(b+c)^{1-q}\\
    =&-b^{1-q}\sum_{c\in I(<j_m)}\sum_{i\geq0}\binom{1-q}{i}\frac{c^i}{b^i}=-b^{1-q}\sum_{i\geq0}b^{-i}\binom{1-q}{i}\sum_{c\in I(<j_m)}c^i.
\end{align*}
On the other hand, if we fix a basis $\{a_i\}_{i=1,\dots,m}$ of $I(<j_m)$,  by Lemma \ref{S_n,k}, we have:
\[\sum_{c\in I(<j_m)}c^i=S_{m,i}(a_1,\dots,a_m)=0\;\;\;\forall i<q^m-1.\]
As elements of $\O_\K\cong\F_q[[u]]\subseteq K[[u]]$, $v\left(\frac{c}{b}\right)\geq1$ for all $c\in I(<j_m)$, and $v(b^{-1})=j_m$, so we get:
\begin{align*}
    &q\cdot v\left((\zeta_{I,m}-\zeta_{I,m-1})(\Xi^{(-1)})\right)=(1-q)v(b)+v\left(\sum_{i\geq q^m-1}\binom{1-q}{i}\sum_{c\in I(<j_m)}\left(\frac{c}{b}\right)^i\right)\\
    \geq&(1-q)v(b)+\min_{\substack{c\in I(<j_m)\\i\geq q^m-1}}\left\{i\cdot v\left(\frac{c}{b}\right)\right\}\geq j_m(q-1)+q^m-1.
\end{align*}
Since $\left(\zeta_{I,m}^{(-1)}\right)(\Xi^{(-1)})=\left(\zeta_{I,m}(\Xi)\right)^\frac{1}{q}=-1$, we get:
\[\sum_{a\in I(j_m)}a^{-1}\otimes a=\zeta_{I,m}-\zeta_{I,m-1}=-\zeta_{I,m}^{(-1)}\cdot(\zeta_{I,m}-\zeta_{I,m-1})(\Xi^{(-1)}).\]
Its valuation, since $\zeta_{I,m}^{(-1)}\in K[[u^\frac{1}{q}]]$, is at least $j_m\frac{q-1}{q}+q^{m-1}-\frac{1}{q}\geq q^{m-1}$, for $m\geq1$.
\end{proof}

\begin{oss}The previous proof makes use of the fact that $I(m)$ is an $\F_q$-vector space of dimension $1$ for $m\gg0$, which is not true if we do not assume $\infty$ to be $\F_q$-rational.\end{oss}

\begin{lemma}\label{k_P}
Fix a point $P\in X(\C)\setminus\{\infty\}$, corresponding to a map $\chi_P:A\to\C$. There is a nonnegative real constant $k_P$ such that $v(\chi_P(a))\geq -k_P\deg(a)$ for all $a\in A$.
\end{lemma}
\begin{proof}
By a purely combinatorial observation, for any set $S\subseteq\mathbb{N}$ there is some $N\in\mathbb{N}$ such that every element of $S$ can be expressed as a finite sum of the elements of $S$ less than $N$. If we take $S$ to be the image of $A\setminus\F_q$ via the degree map, it follows that we can pick a finite set $\{a_1,\dots,a_n\}\subseteq A\setminus\F_q$ such that for all $a\in A\setminus\F_q$ there is a product $a'$ of $a_i$'s with $\deg(a')=\deg(a)$. Since $\deg(a_i)>0$ for all $i$, we can define the following nonnegative real number:
\[k_P:=\max\left\{0,\max_{1\leq i\leq n}\left\{\frac{-v(\chi_P(a_i))}{\deg(a_i)}\right\}\right\},\]
so that $v(\chi_P(a_i))\geq -k_P\deg(a_i)$ for all $i$. We prove the lemma by induction on $\deg(a)$. 

If $\deg(a)=0$ the claim is trivially true. If $\deg(a)>0$ there is a product $a':=\lambda\prod_i a_i^{e_i}$, with $\lambda\in\F_q$, of the same degree and sign, hence $\deg(a-a')<\deg(a)$. We have:
\begin{align*}
    &v(\chi_P(a-a'))\geq -k_P\deg(a-a')\geq-k_P\deg(a)\text{ by inductive hypothesis, since $k_P\geq0$;}\\
    &v(\chi_P(a'))=\sum_i e_i\cdot v(\chi_P(a_i))\geq-\sum_i k_P e_i\cdot\deg(a_i)=-k_P\deg(a')=-k_P\deg(a).
\end{align*}
Hence, $v(\chi_P(a))\geq\min\{v(\chi_P(a')),v(\chi_P(a-a'))\}\geq-k_P\deg(a)$.
\end{proof}

From the previous lemmas we can deduce the following.

\begin{lemma}\label{coefficients of zetas}
For all $k\geq 0$, for all $i\in\frac{1}{q^k}\mathbb{N}$, $\deg((\zeta_I^{(-k)})_{(i)})\leq\log_q(i+1)+k+g+d+1$. For all points $P\in X(\C)\setminus\{\infty\}$, corresponding to maps $\chi_P:A\to\C$, for all $k\geq0$, the following series converges:
\[\sum_{i\geq0}\chi_P\left((\zeta_I^{(-k)})_{(i)}\right)u^i.\]
\end{lemma}
\begin{proof}
The first part of the statement for $k=0$ follows from the inequality of Lemma \ref{(zeta_{I,m})_{(i)}}, using the fact that for all $i$ the sequence $((\zeta_{I,m})_{(i)})_m$ is eventually equal to $(\zeta_I)_{(i)}$. For $k>0$ and $i\in\frac{1}{q^k}\mathbb{N}$, we get:
\[\deg\left((\zeta_I^{(-k)})_{(i)}\right)=\deg\left((\zeta_I)_{(iq^k)}\right)\leq\log_q(iq^k+1)+g+d+1\leq\log_q(i+1)+k+g+d+1.\]

Let's define $k_P$ as in Lemma \ref{k_P}. Then, for all $i>0$ we have:
\[v\left(\chi_P\left((\zeta_I^{(-k)})_{(i)}\right)u^i\right)\geq-k_P\deg\left((\zeta_I^{(-k)})_{(i)}\right)+i\geq i-k_P\log_q(i+1)-k_P(k+g+d+1),\]
which tends to infinity for $i\to\infty$, proving the convergence of $\sum_{i\geq0}\chi_P\left((\zeta_I^{(-k)})_{(i)}\right)u^i$.
\end{proof}

Finally we can prove Proposition \ref{conv eval zeta}.

\begin{proof}[Proof of Proposition \ref{conv eval zeta}]
Define $k_P$ as in Lemma \ref{k_P}. For $m\geq0$, by Lemma \ref{(zeta_{I,m})_{(i)}} we have:
\[v(\zeta_I-\zeta_{I,m})=v\left(\sum_{m'\geq m}\zeta_{I,m'+1}-\zeta_{I,m'}\right)\geq\min_{m'\geq m}v\left(\zeta_{I,m'+1}-\zeta_{I,m'}\right)\geq q^m.\]
For all $i\geq q^m$, by Lemma \ref{coefficients of zetas}, we have:
\begin{align*}
\deg\left((\zeta_I-\zeta_{I,m})_{(i)}\right)&\leq\max\left\{\deg\left((\zeta_I)_{(i)}\right),\deg\left((\zeta_{I,m})_{(i)}\right)\right\}\\
&\leq\max\{\log_q(i+1)+g+d+1,j_m\}=\log_q(i+1)+g+d+1,
\end{align*}
since $j_m\leq m+g+d+1$ and $m\leq\log_q(i+1)$. In particular:
\begin{align*}
    v\left(\sum_i \chi_P\left((\zeta_I-\zeta_{I,m})_{(i)}\right)u^i\right)&=v\left(\sum_{i\geq q^m}\chi_P\left((\zeta_I-\zeta_{I,m})_{(i)}\right)u^i\right)\\
    &\geq\min_{i\geq q^m}\left\{i-k_P\cdot\deg\left((\zeta_I-\zeta_{I,m})_{(i)}\right)\right\}\\
    &\geq\min_{i\geq q^m}\left\{i-k_P(\log_q(i+1)+g+d+1)\right\},
\end{align*}
which tends to infinity for $m\to\infty$. By Proposition \ref{conv eval entire rational}, $\zeta_{I,m}(P)=\sum_i \chi_P\left((\zeta_{I,m})_{(i)}\right)u^i$, hence we get that
\[\lim_m \zeta_{I,m}(P)-\sum_{i\geq0}\chi_P\left((\zeta_I)_{(i)}\right)u^i=\lim_m\left(\sum_i \chi_P\left((\zeta_{I,m}-\zeta)_{(i)}\right)u^i\right)=0. \tag*{\qedhere}\]
\end{proof}

\begin{Def}\label{evaluating zetas}
We define the evaluation of $\zeta_I$ at $P$ as $\zeta_I(P):=\sum_i (\zeta_I)_{(i)}(P)u^i$.
\end{Def}

\begin{cor}\label{zeta zero}
For all $i\geq 1$, we have $\zeta_I(\Xi^{(i)})=0$. Similarly, for all $k\geq0$, for all $i\geq1$, $\sum_j\chi_{\Xi^{(i-k)}}((\zeta_I^{(-k)})_{(j)})u^j=0$ (where $j$ varies among $\frac{1}{q^k}\mathbb{N}$).
\end{cor}
\begin{proof}
For the first identity we use that, for all $i\geq1$, $\zeta_{I,m}(\Xi^{(i)})=0$ for $m\gg0$. For the second identity, note that
\[\left(\sum_{j\in\frac{1}{q^k}\mathbb{N}}\chi_{\Xi^{(i-k)}}((\zeta_I^{(-k)})_{(j)})u^j\right)^{q^k}=\sum_{j\in\mathbb{N}}\chi_{\Xi}^{(i)}((\zeta_I)_{(j)})u^j=0. \tag*{\qedhere}\]
\end{proof}

\subsection{Adjoint Drinfeld modules and adjoint shtuka functions}
From now on, in this section we use the following notation: $V_*:=V_{\bar{I},*}$, $f_*:=f_{\bar{I},*}$, $\zeta:=(\gamma_I\otimes1)\zeta_I$, with $\gamma_I$ defined as in Corollary \ref{gamma_I}, so that $\zeta^{(-1)}=f_*\zeta$.

Proposition \ref{adjoint Drinfeld module} shows a connection between adjoint Drinfeld modules, adjoint shtuka functions, and zeta functions, which is meant to mirror the correspondence between Drinfeld modules, shtuka functions, and special functions (cf. \cite{Thakur}[Eq.($**$)], \cite{Anderson}[Eq.(46)]).

Afterwards, we present some basic definitions and results concerning the coefficients of exponential and logarithmic functions (see for example \cite{Goss}) to prove the interesting Proposition \ref{exp(zeta)}. On the surface the proposition resembles a log-algebraicity result, and could be linked to this rich branch of research (see for example \cite{Anderson}, \cite{Anderson2}, \cite{Starkunits}); on the other hand, it encourages a greater focus on the adjoint exponential function, whose kernel was already studied in works such as \cite{Poonen}.

\begin{prop}\label{adjoint Drinfeld module}
Set $e_m:=\prod_{i=0}^{m-1}f_*^{(-i)}$ for all nonnegative integers $m$. The collection $\{e_m\}_{m\geq0}$ is a basis of the $\C\otimes1$-vector space $H^0(X_\C\setminus\{\infty\},V_*^{(1)})$. 

For all $a\in A$, $1\otimes a$ can be expressed as $\sum_{i=0}^{\deg(a)}(a_i\otimes1)e_i$ with $a_i^{q^i}\in \K$, and the function $\phi^*:A\to \C[\tau^{-1}]$ sending $a$ to $\sum_i a_i\tau^{-i}$ is the adjoint of a normalized Drinfeld module $\phi$ of rank $1$.

Finally, for all $a,b\in A$, $\phi^*_{ab}(\zeta)=(\phi^*_a\circ\phi^*_b)(\zeta)$.
\end{prop}
\begin{proof}
Since $H^0(X_\C\setminus\{\infty\},V_*^{(1)})=\bigcup_{m\geq0}H^0(X_\C,V_*^{(1)}+m\infty)$, for the first part we just need to prove that, for all $m\geq0$, $e_m\in H^0(X_\C\setminus\{\infty\},V_*^{(1)})$ and it has a pole of multiplicity exactly $m$ at $\infty$; using that $\Div(f_*^{(-i)})=V_*^{(-i)}-V_*^{(1-i)}+\Xi^{(-i)}-\infty$, we get:
\[\Div(e_m)=\Div\left(\prod_{i=0}^{m-1}f_*^{(-i)}\right)=V_*^{(1-m)}-V_*^{(1)}+\sum_{i=0}^{m-1}\Xi^{(-i)}-m\infty.\]
If we fix $a\in A$ of degree $m$, $1\otimes a\in H^0(X_\C,V_*^{(1)}+m\infty)$, hence it can be expressed as a sum $\sum_{i=0}^m (a_i\otimes1)e_i$. Moreover, if we twist $k$ times and evaluate at $\Xi$ for all $0\leq k\leq m$ we get the following triangular system of equations in the variables $(a_i)_i$:
\begin{align*}
    &\left\{a=\sum_{i=0}^k\left(a_i^{q^k}\prod_{j=k-i}^k f_*^{(j)}(\Xi)\right)\right\}_k\\
    \Longrightarrow&\left\{a_k^{q^k}=\left(\prod_{j=0}^k f_*^{(j)}(\Xi)\right)^{-1}\left(a-\sum_{i=0}^{k-1}a_i^{q^k}\prod_{j=k-i}^k f_*^{(j)}(\Xi)\right)\right\}_k
\end{align*}.
From this system we can deduce that $a_0=a$ and, since $f_*^{(j)}(\Xi)\in \K$ for all $j\geq0$, that $a_k^{q^k}\in \K$ for all $k$. Finally, since $\deg(a)=\deg(e_m)$, and $\sgn(f_*^{(i)})=\sgn(f_*)=1$ for all $i\geq0$, the sign of $e_m^{(i)}$ is also $1$ for all $m,i$, and we have:
\[\sgn(a)=\sgn\left(\sum_{i=0}^m (a_i^{q^m}\otimes1)e_i^{(m)}\right)=\sgn((a_m^{q^m}\otimes1)e_m^{(m)})=a_m^{q^m}\sgn(e_m^{(m)})=a_m^{q^m},\]
so $a_m=\sgn(a)$. For all $a\in A$, write $\phi^*_a:=\sum_i a_i\tau^{-i}$. Since for all $k\geq 0$ and for all $a\in A$ we have $\zeta e_k=\zeta^{(-k)}$ and $1\otimes a=(1\otimes a)^{(-k)}=\sum_i (a_i^\frac{1}{q^k}\otimes 1)e_i^{(-k)}$, we get the following equations for all $k\geq0$ and $a,b\in A$:

\begin{align*}
(1\otimes a)\zeta^{(-k)}&=\sum_i (a_i^\frac{1}{q^k}\otimes 1)(e_i\zeta)^{(-k)}=\sum_i (a_i^\frac{1}{q^k}\otimes 1)\zeta^{(-k-i)}=\tau^{-k}\circ\phi_a^*(\zeta);\\
    \phi_{ab}^*(\zeta)&=(1\otimes ab)\zeta=(1\otimes a)\left((1\otimes b)\zeta\right)=\sum_i(1\otimes a)\left((b_i\otimes 1)\zeta^{(-i)}\right)\\
    &=\sum_i(b_i\otimes 1)\left((1\otimes a)\zeta^{(-i)}\right)=\sum_i(b_i\otimes 1)\left(\tau^{-i}\circ\phi_a^*(\zeta)\right)=\left(\phi_b^*\circ\phi_a^*\right)(\zeta).
\end{align*}
Since the elements $(\zeta^{(-i)})_{i\geq0}=(\zeta e_i)_{i\geq0}$ are all $\C\otimes1$-linearly independent, we have the equality $\phi_{ab}^*=\phi_b^*\circ\phi_a^*$. Together with the fact that $\deg(\phi^*_a)=\deg(a)$ and $a_{\deg(a)}=\sgn(a)$, this means that the function $\phi:=(\phi^*)^*:A\to \K[\tau]$ is a normalized Drinfeld module of rank $1$.
\end{proof}

From this point onwards, $\phi$ and $\phi^*$ are defined as in Proposition \ref{adjoint Drinfeld module}.

\begin{Def}\label{definition Lambda'}
    We call $\Lambda'\subseteq\C$ the period lattice of $\phi$, and define $\exp:=\exp_{\Lambda'}$ and $\log:=\log_{\Lambda'}$. We call $\Lambda\subseteq K$ the unique fractional ideal isogenous to $\Lambda'$ such that $\Lambda(\leq0)=\F_q$. We choose a nonzero element of least norm $\tilde{\pi}_\phi\in\Lambda'$, and we call it the \emph{fundamental period} of $\Lambda'$; for simplicity we denote it by $\tilde{\pi}$ for the rest of the section. 
\end{Def}

\begin{oss}
Since $\operatorname{rk}(\Lambda')=1$, a lattice isogenous to $\Lambda'$ is uniquely determined by its nonzero elements of least norm, hence the assumption that for $\Lambda$ they are $\F_q$ is enough to imply $\Lambda\subseteq K$.
Our choice of $\tilde{\pi}$ is up to a factor in $\F_q^\times$, and we have $\tilde{\pi}\Lambda=\Lambda'$.
\end{oss}

\begin{Def}
    Define $\exp^*:=\sum_i e_i^\frac{1}{q^i}\tau^{-i}\in\C[[\tau^{-1}]]$, where for all $i$ $e_i$ is the $i$-th coefficient of $\exp\in\C[[\tau]]$. We call $\exp^*$ the \emph{adjoint exponential function}.
\end{Def}

\begin{oss}\label{exp property}
    Since $\exp\circ(a\tau^0)=\phi_a\circ\exp$ for all $a\in A$, we easily deduce the following identity in $\C[[\tau^{-1}]]$ for all $a\in A$:
    \[a\exp^*=\exp^*\circ\phi^*_a.\]
\end{oss}

\begin{oss}\label{c_k in K}
    If we write $\exp:=\sum_i e_i\tau^i$ and $\phi_a:=\sum_j a_j\tau^j$ for some $a\in A$ (with $a_0=a$), the equation $\exp\circ (a\tau^0)=\phi_a\circ\exp$ becomes:
    \[\sum_k (e_k a^{q^k})\tau^k=\sum_k\left(\sum_{i+j=k} a_j e_i^{q^j}\right)\tau^k\Rightarrow e_k(a^{q^k}-a)=\sum_{i=0}^{k-1}a_{k-i}e_i^{q^{k-i}}.\]
    Since $e_0=1$, we get that $e_k\in \K$ for all $k\geq0$ by induction.
\end{oss}

\begin{Def}
For any rank $1$ projective $A$-module $L\subseteq\C$, we define, for all $k\geq1$:
\begin{align*}
    &S_k(L):=\sum_{\substack{\lambda_1,\dots,\lambda_k\in L\setminus\{0\}\\i\neq j\Rightarrow \lambda_i\neq \lambda_j}} (\lambda_1\cdots \lambda_k)^{-1};
    &P_k(L):=\sum_{\lambda\in L\setminus\{0\}} \lambda^{-k}.
\end{align*}
We also set $S_0(L):=1$ and $P_0(L):=-1$.
\end{Def}
\begin{oss}
    By definition $\exp_L(x):=x\prod_{\lambda\in L\setminus\{0\}}\left(1-\frac{x}{\lambda}\right)\in\C[[x]]$, and by absolute convergence we can expand the product and rearrange the terms of the series, so we get $\exp_L(x)=\sum_{i\geq0}S_i(L)x^{i+1}$; in particular, if $i+1$ is not a power of $q$, $S_i(L)=0$. 
\end{oss}

\begin{oss}\label{c_i valuation}
Note that in the summation that defines $S_{q^i-1}(\tilde{\pi}\Lambda)$ there is a unique summand of greatest norm, given by the inverse of the product of the $q^i-1$ nonzero elements of lower norm of $\tilde{\pi}\Lambda$. Since $\F_q^\times$ are the nonzero elements of lowest degree of $\Lambda$, this product has valuation at least:
\[\sum_{j=0}^{i-1} (q^{j+1}-q^j)(j+v(\tilde{\pi}))=i q^i-\left(\sum_{j=0}^{i-1} q^j \right)+q^i v(\tilde{\pi})\geq(i-1+v(\tilde{\pi}))q^i.\]

In particular, since $\lim_i \frac{1}{q^i}v\left(S_{q^i-1}(\tilde{\pi}\Lambda)\right)=\infty$, $\exp:\C\to\C$ is an entire function with an infinite radius of convergence, and $\exp^*:\C\to\C$, while not being a power series, is continuous, converges everywhere, and sends $0$ to $0$; moreover, $\lim_{\|z\|\to0}\exp^*(z)z^{-1}=1$.
\end{oss}
\begin{oss}
Since they are continuous $\F_q$-linear endomorphisms of $\C$, we can extend uniquely both $\exp$ and $\exp^*$ to continuous $\F_q\otimes K$-linear endomorphisms of $\C\hat\otimes K$.
\end{oss}

\begin{lemma}\label{log_L}
For all rank $1$ lattices $L\subseteq\C$, $\log_L=-\sum_i P_{q^i-1}(L)\tau^i$.
\end{lemma}
\begin{proof}
For all $i\geq1$, $-iS_i(L)=\sum_{j=0}^{i-1}(-1)^{i-j}S_j(L)P_{i-j}(L)$ by Newton's identities. Setting $i=q^k-1$ with $k\geq1$, since $S_j(L)=0$ if $j+1$ is not a power of $q$, we get:
\[S_{q^k-1}(L)=\sum_{j=0}^{k-1}S_{q^j-1}(L)P_{q^k-q^j}(L)=\sum_{j=0}^{k-1}S_{q^j-1}(L)(P_{q^{k-j}-1}(L))^{q^j}.\]
In particular:
\begin{align*}
    \exp_L\circ\left(-\sum_{i\geq0}P_{q^i-1}(L)\tau^i\right)&=\sum_{k\geq0}\left(-\sum_{j=0}^k S_{q^j-1}(L)(P_{q^{k-j}-1}(L))^{q^j}\right)\tau^k\\
    &=-S_0(L) P_0(L)=1.
\end{align*}
The uniqueness of right inverses proves the thesis.
\end{proof}

\begin{prop}\label{exp(zeta)}
We have the following functional identity:
\[\exp^*(\zeta)=0.\]
\end{prop}
\begin{proof}
By Remark \ref{exp property}, for all $a\in A$ we have $\exp^*\circ\phi_a^*=(a\otimes 1)\exp^*$ as endomorphisms of $\C\hat\otimes K$; by Proposition \ref{adjoint Drinfeld module}, $\phi^*_a(\zeta)=(1\otimes a)\zeta$ . Hence, for all $a\in A$:
\[0=\exp^*(0)=\exp^*(\phi^*_a(\zeta)-(1\otimes a)\zeta)=(a\otimes1-1\otimes a)\exp^*(\zeta).\]
For $a\not\in\F_q$, $a\otimes1-1\otimes a$ is invertible in $\C\hat\otimes K$, with inverse $\sum_{i\geq0}a^{-i-1}\otimes a^i$, so we get the thesis.
\end{proof}

\subsection{The fundamental period \texorpdfstring{$\tilde{\pi}$}{}}

Finally, in this subsection we are able to link the zeta function $\zeta_I$ and the fundamental period $\tilde{\pi}$. Fix an element $a_I\in I\setminus\{0\}$ of least degree.

\begin{prop}\label{Lambda=I}
The $A$-modules $a_I^{-1}I$ and $\Lambda$ coincide as submodules of $\C$.
\end{prop}
\begin{proof}
    Since the nonzero elements of least degree of both $a_I^{-1}I$ and $\Lambda$ are $\F_q^\times$, it suffices to show that $I$ and $\Lambda$ are isogenous. Let's first give an intuitive rundown of the proof.
    
    For all $n\geq0$, for all $k\geq0$, if $n<k$ then $\zeta^{(-k)}(\Xi^{(-n)})=(\zeta(\Xi^{(k-n)}))^{\frac{1}{q^k}}=0$ by Corollary \ref{zeta zero}, while if $n\geq k$ then:
    \begin{align*}
        \zeta^{(-k)}(\Xi^{(-n)})=&\gamma_I^\frac{1}{q^k}\sum_{a\in I\setminus\{0\}}a^{\frac{1}{q^n}-\frac{1}{q^k}}=\left(\gamma_I\sum_{a\in I\setminus\{0\}}\left(\frac{a}{\gamma_I}\right)^{1-q^{n-k}}\right)^\frac{1}{q^n}\\
        =&\left(\gamma_I P_{q^{n-k}-1}(\gamma_I^{-1}I)\right)^\frac{1}{q^n}.
    \end{align*}
    Since $\exp^*(\zeta)=0$ by Proposition \ref{exp(zeta)}, evaluating $\exp_*(\zeta)$ at $\Xi^{(-n)}$ we should get:
    \begin{align*}
        0=\exp_*(\zeta)(\Xi^{(-n)})&=\sum_{k\geq0}S_{q^k-1}(\Lambda')^{\frac{1}{q^k}}\zeta^{(-k)}(\Xi^{(-n)})\\
        &=\sum_{0\leq k\leq n}S_{q^k-1}(\Lambda')^{\frac{1}{q^k}}\left(\gamma_I P_{q^{n-k}-1}(\gamma_I^{-1}I)\right)^\frac{1}{q^n}=\\
        &=\left(\gamma_I\sum_{0\leq k\leq n}P_{q^{n-k}-1}(\gamma_I^{-1}I)S_{q^k-1}(\Lambda')^{q^{n-k}}\right)^{\frac{1}{q^n}},
    \end{align*}
    which by Lemma \ref{log_L} implies that $\log_{\gamma_I^{-1}I}\circ\exp=1$. In particular, $\exp=\exp_{\gamma_I^{-1}I}$, therefore their zero loci are the same, which means that $\gamma_I^{-1}I=\Lambda'=\tilde{\pi}\Lambda$.

The previous reasoning is not rigorous only when it assumes that evaluation at $\Xi^{(-n)}$ commutes with the expansion of $\exp_*(\zeta)$, therefore to prove the theorem it suffices to show that $\sum_{0\leq k\leq n}S_{q^k-1}(\Lambda')^{\frac{1}{q^k}}\zeta^{(-k)}(\Xi^{(-n)})=0$.
For all $k\in\mathbb{N}$, define $c_k:=S_{q^k-1}(\tilde{\pi}\Lambda)^\frac{1}{q^k}$; by Remark \ref{c_k in K}, $c_k\in\F_q((u^\frac{1}{q^k}))$.
By Lemma \ref{coefficients of zetas} and Remark \ref{c_i valuation} respectively, we have the following inequalities for all $i\in\frac{1}{q^k}\mathbb{N}$, for all $k\in\mathbb{N}$:
\begin{align*}
&\deg\left((\zeta^{(-k)})_{(i)}\right)\leq\log_q(i+1)+k+g+d+1, &v\left(c_k\right)\geq k-1+v(\tilde{\pi})=:k'.
\end{align*}

Fix a positive integer $n$. Since $\exp^*(\zeta)=0$, for any arbitrarily large $N$ we can choose a positive integer $m\geq n$ such that $v\left(\sum_{k=0}^m c_k\zeta^{(-k)}\right)\geq N$.
For all $k\leq m$ we can write the following, where the index $i$ varies among $\frac{1}{q^m}\Z$:
\[c_k=\sum_{i\geq k'}\lambda_{k,i} u^i\in\F_q\left(\left(u^{\frac{1}{q^m}}\right)\right)\text{ with }\lambda_{k,i}\in\F_q.\]
Let's rearrange $\sum_{k=0}^m c_k\zeta^{(-k)}$, with the indexes $i$ and $j$ varying among $\frac{1}{q^m}\mathbb{Z}$:
\begin{align*}
    \sum_{k=0}^m\sum_{j\geq k'} \lambda_{k,j}\zeta^{(-k)}u^j&=\sum_{k=0}^m\sum_{j\geq k'}\lambda_{k,j}\sum_{i\geq0}\left(\zeta^{(-k)}\right)_{(i)}u^{i+j}\\
    &=\sum_{i\geq0}\left(\sum_{k=0}^m\sum_{j=k'}^i\lambda_{k,j}\left(\zeta^{(-k)}\right)_{(i-j)}\right)u^i.
\end{align*}
Since $v\left(\sum_{k=0}^m c_k\zeta^{(-k)}\right)\geq N$, we get that, for $i\in\frac{1}{q^m}\Z$ and $i<N$:
\[\sum_{k=0}^m\sum_{j=k'}^i\lambda_{k,j}\left(\zeta^{(-k)}\right)_{(i-j)}=0.\]
Using this result and Lemma \ref{coefficients of zetas}, the evaluation $\sum_{k=0}^m c_k\zeta^{(-k)}(\Xi^{(-n)})$ can be rearranged as follows:
\begin{align*}
\sum_{k=0}^m\sum_{j\geq k'}\lambda_{k,j}\zeta^{(-k)}(\Xi^{(-n)})u^j
=&\sum_{k=0}^m\sum_{j\geq k'}\lambda_{k,j}\sum_{i\geq0}\left(\zeta^{(-k)}\right)_{(i)}(\Xi^{(-n)}) u^{i+j}\\
=&\sum_{i\geq0}\left(\sum_{k=0}^m\sum_{j=k'}^i\lambda_{k,j}\left(\zeta^{(-k)}\right)_{(i-j)}\right)(\Xi^{(-n)})u^i\\
=&\sum_{i\geq N}\left(\sum_{k=0}^m\sum_{j=k'}^i\lambda_{k,j}\left(\zeta^{(-k)}\right)_{(i-j)}\right)(\Xi^{(-n)})u^i.
\end{align*}
For $i-j,k\geq0$, since $j\geq k'\geq v(\tilde{\pi})-1$, and since $\log_q(x)\leq x$ for all $x>0$, we have:
\[\deg\left((\zeta^{(-k)})_{(i-j)}\right)\leq \log_q(i-j+1)+k+g+d+1\leq i+k+g+d+3-v(\tilde{\pi})=:i+C,\]
so each summand has valuation at least $i-\frac{i+C}{q^n}\geq N-\frac{N+C}{q^n}$, which tends to infinity as $N$ tends to infinity. Since $m=m(N)$ depends on $N$ and tends to infinity as $N$ does, we have: 
\[0=\lim_{N\to\infty}\sum_{k=0}^{m(N)} c_k\zeta^{(-k)}(\Xi^{(-n)})=\lim_{m\to\infty}\sum_{k=0}^m c_k\zeta^{(-k)}(\Xi^{(-n)})=\sum_{k=0}^n c_k\zeta^{(-k)}(\Xi^{(-n)}),\] 
where we used that $\zeta^{(-k)}(\Xi^{(-n)})=0$ for $k>n$ by Corollary \ref{zeta zero}. This concludes the proof.
\end{proof}

\begin{prop}\label{a_I/pi}
The following identity holds in $\C\hat\otimes K$:
\[\frac{\left((a_I\tilde{\pi}^{-1}\otimes1)\zeta_I\right)^{(-1)}}{(a_I\tilde{\pi}^{-1}\otimes1)\zeta_I}=f_*,\]
\end{prop}
\begin{proof}
From the definition of $\gamma_I$ we have $\frac{\zeta_I}{\zeta_I^{(1)}}=(\gamma_I\otimes1)^{q-1}f_*^{(1)}$. Since $\Lambda=a_I^{-1}I$ and $\tilde{\pi}\Lambda=\gamma_I^{-1}I$, we deduce $\gamma_I=\frac{a_I}{\tilde{\pi}}$ up to a factor in $\F_q^\times$.
\end{proof}
\begin{cor}\label{pi in K}The element $\tilde{\pi}^{q-1}$ is contained in $\K$.
\end{cor}
\begin{proof} The element $\frac{(a_I\otimes1)^{q-1}f_*^{(1)}\zeta_I^{(1)}}{\zeta_I}\in\K\hat\otimes K$ is equal to $(\tilde{\pi}\otimes1)^{q-1}$.\end{proof}

Equivalently, we can say that $(a_I\tilde{\pi}^{-1}\otimes1)\zeta_I$ is an ``eigenfunction" for the dual Drinfeld module $\phi^*$, in the following sense.

\begin{teo}\label{dual special function}
    The following identity holds for all $a\in A$:
    \[\phi^*_a\left((a_I\tilde{\pi}^{-1}\otimes1)\zeta_I\right)=(1\otimes a)(a_I\tilde{\pi}^{-1}\otimes1)\zeta_I.\]
\end{teo}
\begin{proof}
    In the proof of Proposition \ref{adjoint Drinfeld module}, we showed that, for all $a\in A$, $(1\otimes a)\zeta=\phi^*_a(\zeta)$, where $\zeta=(\gamma_I\otimes1)\zeta_I=(a_I\tilde{\pi}^{-1}\otimes1)\zeta_I$ up to a factor in $\F_q^\times$.
\end{proof}

We can finally state and prove a stronger version of Theorem \ref{functional identity weak}.

\begin{teo}\label{functional identity}
    The following functional identity is well posed and true in $\K\hat\otimes K$:
    \[\zeta_I=-(a_I^{-1}\otimes a_I)\prod_{i\geq0}\left((\tilde{\pi}^{1-q}\otimes1)f_*^{(1)}\right)^{(i)}.\]
\end{teo}
\begin{proof}
    Using the partial version of this theorem, we deduce the following identity in $\O_\K\hat\otimes K$:
    \[\zeta_I=-(a_I^{-1}\otimes a_I)\prod_{i\geq0}\left((\lambda^{1-q}\otimes1)f'_{\bar{I},*}{}^{(1)}\right)^{(i)},\]
    where $f'_{\bar{I},*}$ is a scalar multiple of $f_{\bar{I},*}$, and $\lambda\in\O_\K$ is some constant. We deduce:
    \[\frac{\zeta_I}{\zeta_I^{(1)}}=(a_I^{q-1}\otimes 1)(\lambda^{1-q}\otimes1)f'_{\bar{I},*}{}^{(1)}.\]
    On the other hand, by Proposition \ref{a_I/pi}, we know that
    \[\frac{\zeta_I}{\zeta_I^{(1)}}=\left(\frac{a_I}{\tilde{\pi}}\otimes1\right)^{q-1}f_*^{(1)},\]
    hence $(\lambda^{1-q}\otimes1)f'_{I,*}{}^{(1)}=(\tilde{\pi}^{1-q}\otimes1)f_*^{(1)}$ and we get the desired identity.
\end{proof}

\subsection{A stronger restatement of the main theorem}

Using the results of the previous subsection, and Remark \ref{oss f_I}, we can prove Theorem \ref{Sf module}, as stated in the introduction. It is a stronger version of Theorem \ref{Sf module weak} with an explicit proportionality constant and with a more natural dependence on the period lattice.

First, we prove the following results.

\begin{lemma}\label{lemma h_I,J}
    Fix a nonzero ideal $J<A$, with degree $d_J$. Then, for all ideal classes $\bar{I}\in Cl(A)$, there is some rational function $h_{J,\bar{I}}$ on $X_H$ with sign $1$ and divisor:
    \[\Div(h_{J,\bar{I}})=V_{\bar{I},*}^{(1)}+V_{\bar{J}+\bar{I}}-J-\Xi-(2g-d_J-1)\infty.\]
    Moreover, the functions $\{h_{J,\bar{I}}\}_{\bar{I}\in Cl(A)}$ are all conjugated by the action of $\G(H/K)$.
\end{lemma}
\begin{proof}
    Fix some nonzero ideal $I < A$ with ideal class $\bar{I}$, and call $d_I$ its degree. Define $D:=J+\Xi+(2g-d_J-1)\infty$ and consider the divisor $D-V_{\bar{I},*}^{(1)}$: we want to prove that it is linearly equivalent to $V_{\bar{J}+\bar{I}}$. First of all, its degree is $g$, hence it is linearly equivalent to some effective divisor $W$. Moreover, we have the following linear equivalences:
    \begin{align*}
        \red_\K(W)&\sim\red_\K(D)-\red_\K(V_{\bar{I},*})
        \sim (J+I)+(g-d_J-d_I)\infty\sim \red_\K(V_{\bar{J}+\bar{I}});\\
        W-W^{(1)}&\sim(D-D^{(1)})-(V_{\bar{I},*}-V_{\bar{I},*}^{(1)})^{(1)}\sim(\Xi-\Xi^{(1)})-(\infty-\Xi)^{(1)}\sim\Xi-\infty\\
        &\sim V_{\bar{J}+\bar{I}}-V_{\bar{J}+\bar{I}}^{(1)}.
    \end{align*}
By Proposition \ref{con}, the two conditions imply that $W=V_{\bar{J}+\bar{I}}$.

By Remark \ref{Hayes}, the divisors $\{V_{\bar{I},*}^{(1)}+V_{\bar{I}+\bar{J}}-D\}_{\bar{I}\in Cl(A)}$ are $H$-rational. Moreover, by the same reasoning as Remarks \ref{Hayes} and \ref{Hayes2}, they are all conjugated by the action of $\G(H/K)$, therefore for all $\bar{I}\in Cl(A)$ there is a unique function $h_{J,\bar{I}}\in H(X)$ with divisor $V_{\bar{I},*}^{(1)}+V_{\bar{I}+\bar{J}}-D$ and sign $1$, and they are all conjugated by $\G(H/K)$ up to a scalar factor. We just need to prove that for all $\sigma\in\G(H/K)$, $\sgn(h_{J,\bar{A}}^\sigma)=1$. Fix some $a\in J$ of positive degree and sign $1$, and define $s:=(1\otimes a)(1-a^{-1}\otimes a)h_{J,\bar{A}}$: it suffices to prove that $\sgn(s^\sigma)=1$ for all $\sigma\in\G(H/K)$. Since $s$ has poles only at $\infty$, it can be written as a finite sum $\sum_{i=0}^d h_i\otimes a_i\in H\otimes A$ for some nonnegative integer $d$, where the $a_i$'s have sign $1$ and strictly increasing degree, and $h_d\neq0$. Since $\sgn(s)=1$, we deduce $h_d=1$, hence for all $\sigma\in\G(H/K)$ $\sgn(s^\sigma)=h_d^\sigma=1$.
\end{proof}

\begin{prop}\label{quotient of zetas}
    For any pair of ideals $I,J<A$, with degrees respectively $d_I$ and $d_J$, the quotient $\frac{\zeta_I}{\zeta_J}$ is contained in $H(X)\subseteq\K\hat\otimes K$, with divisor $V_{\bar{I},*}^{(1)}-V_{\bar{J},*}^{(1)}+I-J+(d_J-d_I)\infty$. Moreover, $\red_u\left(\sgn\left(\frac{\zeta_I}{\zeta_J}\right)\right)=1$.
\end{prop}
\begin{proof}
    Recall that for all integers $m\gg0$ there is a function $\delta_{\bar{J},m}\in\K(X)$ with divisor $V_{\bar{J},m}+V_{\bar{J},*,m}-2g\infty$. Since the sequence $(V_{\bar{J},m}+V_{\bar{J},*,m})_m$ converges to $V_{\bar{J}}^{(1)}+V_{\bar{J},*}^{(1)}$ in $X^{[2g]}(\K)$, by Proposition \ref{convergence of functions and divisors} we can choose each $\delta_{\bar{J},m}$ so that the sequence $(\delta_{\bar{J},m})_m$ converges to $\delta_{\bar{J}}^{(1)}$ in $K((u))$. In $K((u))$, $\frac{\zeta_I}{\zeta_J}\delta_{\bar{J}}^{(1)}$ is the limit of the sequence $\left(\frac{\zeta_{I,m}}{\zeta_{J,m}}\delta_{\bar{J},m}\right)_m$, and the divisor of the $m$-th element of the sequence is:\[V_{\bar{I},*,m}-V_{\bar{J},*,m}+I-J+(d_J-d_I)\infty+\Div(\delta_{\bar{J},m})=V_{\bar{I},*,m}+V_{\bar{J},m}+I-J-(2g+d_I-d_J)\infty.\] By Proposition \ref{convergence of functions and divisors}, $\frac{\zeta_I}{\zeta_J}\delta_{\bar{J}}^{(1)}$ is rational, with divisor:\[\left(\lim_m (V_{\bar{I},*,m}+V_{\bar{J},m}+I+d_J\infty)\right)-J-(2g+d_I)\infty=V_{\bar{I},*}^{(1)}+V_{\bar{J}}^{(1)}+I-J-(2g+d_I-d_J)\infty.\]In particular, $\frac{\zeta_I}{\zeta_J}$ is rational, with divisor:
    \[V_{\bar{I},*}^{(1)}+V_{\bar{J}}^{(1)}+I-J-(2g+d_I-d_J)\infty-\Div(\delta_{\bar{J}}^{(1)})=V_{\bar{I},*}^{(1)}-V_{\bar{J},*}^{(1)}+I-J+(d_J-d_I)\infty.\]
    Since the divisor is $H$-rational and $\frac{\zeta_I}{\zeta_J}(\Xi)=\frac{\zeta_I(\Xi)}{\zeta_J(\Xi)}=1$, $\frac{\zeta_I}{\zeta_J}\in H(X)$. To study the sign, note first that if we have a sequence $(h_m=\sum_{i=0}^k c_{m,i}\otimes a_i)_m$ converging to $h=\sum_{i=0}^k c_i\otimes a_i$ in $\C\otimes A$, where the $a_i$'s are independent and with strictly increasing degree, and the degree of $h_m$ at $\infty$ is eventually equal to the degree of $h$ at $\infty$ (i.e. $c_k\neq0$), we have that $\lim_m\sgn(h_m)=\lim_m c_{m,k}\sgn(a_k)=c_k\sgn(a_k)=\sgn(h).$ In particular, we deduce that $\lim_m\sgn(\delta_{\bar{J},m})=\sgn(\delta_{\bar{J}})^q=1$. If we fix any nonzero element $c\in J$, of degree $d_c$, both the rational functions $\left(\frac{\zeta_{I,m}}{\zeta_{J,m}}\delta_{\bar{J},m}(1\otimes c)\right)_m$ and their limit $\frac{\zeta_I}{\zeta_J}\delta_{\bar{J}}^{(1)}(1\otimes c)$ have only a pole at $\infty$, of degree $2g+d_I+(d_c-d_J)$, hence they belong to $\C\otimes A$: we deduce that $\sgn\left(\frac{\zeta_I}{\zeta_J}\right)=\lim_m\sgn\left(\frac{\zeta_{I,m}}{\zeta_{J,m}}\right)$. Fix an $\F_q$-basis $(a_i)_i$ of $I$, with strictly increasing degrees and sign $1$. We have:
    \begin{align*}\sgn(\zeta_{I,m})=&\sum_{a\in I(j_m)}a^{-1}\sgn(a)=-\sum_{a\in I(<j_m)}(a_{m+1}-a)^{-1}=-a_{m+1}^{-1}\sum_{i\geq0}\sum_{a\in I(<j_m)}\frac{a^i}{a_{m+1}^i}\\
    =&-a_{m+1}^{-1}\sum_{i\geq0}\sum_{e_1+\cdots+e_m=i}\lambda(e_1,\dots,e_m)\left(\frac{a_1}{a_{m+1}}\right)^{e_1}\cdots\left(\frac{a_m}{a_{m+1}}\right)^{e_m},\end{align*} 
    where $\lambda(e_1,\dots,e_m)\in\F_q$ is a certain coefficient. By Lemma \ref{S_n,k}, if $\lambda(e_1,\dots,e_m)\neq0$, we must have $e_j\geq q^j-q^{j-1}$ for $j=1,\dots,m$. In particular, since the norms of the elements $\left(\frac{a_j}{a_{m+1}}\right)_j$ are strictly increasing and less than 1, the unique summand of maximum norm corresponds to the $m$-uple $(e_j)_j=(q^j-q^{j-1})_j$. Since $\lambda((q^j-1)_j)=\binom{q^m-1}{q-1,\dots,q^m-q^{m-1}}=1$, we get that the first term in the expansion of $\sgn(\zeta_{I,m})$ in $\F_q[[u]]$ is $-u^M$ for some integer $M$. Since the same argument holds for $\zeta_{J,m}$, we obtain that $\sgn\left(\frac{\zeta_I}{\zeta_J}\right)=u^N+o(u^N)$ for some integer $N$.
\end{proof}

\begin{teo}\label{Sf module}
Let $\phi$ be a normalized Drinfeld module of rank $1$ with period lattice $\rho_I I$, where $\rho_I\in\mathbb{C}_\infty^\times$ and $I<A$ is a nonzero ideal. Fix an ideal $J<A$ of degree $d_J$ such that $J\Omega\cong A$, and denote by $h\in H(X)^\times$ the unique function with $\sgn(h)=1$ and $\Div(h)=V_{\bar{I},*}^{(1)}+V_{\bar{I}+\bar{J}}-J-\Xi-(2g-d_J-1)\infty$. The following $A$-submodules of $\C\hat\otimes A$ coincide:
\[\Sf(\phi)=\frac{(\rho_I\otimes1)h}{\zeta_I}(\F_q\otimes IJ).\]
\end{teo}
\begin{proof}
    Let's denote by $f$ the shtuka function relative to $\phi$ and $f_*$ the dual shtuka function relative to $\phi^*$.
    By Definition \ref{definition Lambda'},if we fix $a_I\in I$ of least degree, we have the equality $a_I\rho_I=\tilde{\pi}_\phi$ up to a factor in $\F_q^\times$, hence by Proposition \ref{a_I/pi} we have $\frac{\left((\rho_I^{-1}\otimes1)\zeta_I\right)^{(-1)}}{(\rho_I^{-1}\otimes1)\zeta_I}=f_*$. On the other hand, by Remark \ref{oss f_I} and Proposition \ref{Lambda=I} respectively, we have the identities $\Div(f)=V_{\bar{I}+\bar{J}}^{(1)}-V_{\bar{I}+\bar{J}}+\Xi-\infty$ and $\Div(f_*)=V_{\bar{I},*}-V_{\bar{I},*}^{(1)}+\Xi-\infty$, hence:
    \[\Div\left(\frac{h^{(1)}}{h}\right)=V_{\bar{I},*}^{(2)}+V_{\bar{I}+\bar{J}}^{(1)}-\Xi^{(1)}-V_{\bar{I},*}^{(1)}-V_{\bar{I}+\bar{J}}+\Xi=\Div\left(\frac{f}{f_*^{(1)}}\right);\]
    since the rational functions $\frac{h^{(1)}}{h}$ and $\frac{f}{f_*^{(1)}}$ both have sign $1$, they coincide. In particular, we have that:
    \[\left(\frac{(\rho_I\otimes1)h}{\zeta_I}\right)^{(1)}\left(\frac{(\rho_I\otimes1)h}{\zeta_I}\right)^{-1}=\frac{h^{(1)}}{h}f_*^{(1)}=f,\]
    hence $\frac{(\rho_I\otimes1)h}{\zeta_I}\in (\F_q\otimes K)\Sf(\phi)$.
    By Theorem \ref{Sf module weak}, the $A$-module $\Sf(\phi)\subseteq\C\hat\otimes A$ coincides with $(\F_q\otimes IJ)\frac{\delta_{\bar{I}+\bar{J}}(\lambda\otimes1)}{f_{\bar{I}+\bar{J},*}\zeta_{IJ}}$, where $\lambda\in\C$ is some nonzero constant. To conclude the proof we just need to show that the product $\left(\frac{\delta_{\bar{I}+\bar{J}}}{f_{\bar{I}+\bar{J},*}\zeta_{IJ}}\right)\left(\frac{h}{\zeta_I}\right)^{-1}$ is a constant, i.e. it is a rational function with trivial divisor. By Proposition \ref{quotient of zetas}, $\frac{\zeta_{IJ}}{\zeta_I}$ is a rational function, and we get the following:
    \begin{align*}
        &\Div\left(\frac{\delta_{\bar{I}+\bar{J}}}{f_{\bar{I}+\bar{J},*}\zeta_{IJ}}\frac{\zeta_I}{h}\right)=\Div(\delta_{\bar{I}+\bar{J}})-\Div(f_{\bar{I}+\bar{J},*})+\Div\left(\frac{\zeta_{I}}{\zeta_{IJ}}\right)-\Div(h)\\
        =&(V_{\bar{I}+\bar{J}}+V_{\bar{I}+\bar{J},*}-2g\infty)+(-V_{\bar{I}+\bar{J},*}+V_{\bar{I}+\bar{J},*}^{(1)}-\Xi+\infty)+\\
        +&(V_{\bar{I},*}^{(1)}-V_{\bar{I}+\bar{J},*}^{(1)}-J+d_J\infty)+(-V_{\bar{I},*}^{(1)}-V_{\bar{I}+\bar{J}}+J+\Xi+(2g-d_J-1)\infty)=0. \tag*{\qedhere}\end{align*}
\end{proof}
In the notation of the previous theorem and of Lemma \ref{lemma h_I,J} we give the next definition.

\begin{Def}\label{pseudocanonical special function}
    Fix an element $a_I\in I$ of least degree. We define the \emph{pseudocanonical special function} as follows:\[\omega_{\phi,J}:=-\frac{(\pi_\phi\otimes1)h_{J,\bar{I}}}{\zeta_I}(a_I^{-1}\otimes a_I)\in\C\hat\otimes K.\]
\end{Def}
\begin{oss}
    Since $\tilde{\pi}_\phi=\rho_I a_I$ up to a factor in $\F_q^\times$, by Theorem \ref{Sf module} $\omega_{\phi,J}$ belongs to $\Sf(\phi)(\F_q\otimes K)$, and is well defined up to a factor in $\F_q^\times$. While $\omega_{\phi,J}$ depends on the choice of the ideal $J$, it does not depend on the choice of $I$ but only on its class.
\end{oss}

\section{An identity involving special functions and zeta functions \`a la Anderson}\label{Anderson zeta}

In this last section, we prove the generalization of \cite[Thm. 7.3]{Green} in the form of Theorem \ref{xi equation}. We define a zeta function ``\`a la Anderson" $\xi_\phi$ relative to a Drinfeld module $\phi$ of rank 1, and then relate it to a pseudocanonical special function.

Fix a normalized Drinfeld module of rank $1$ $\phi:A\to \K[\tau]$ with period lattice $\rho_I I$, where $I<A$ is a nonzero ideal -- so that its Drinfeld divisor is $V_{\bar{I}-\bar{\Omega}}$ by Remark \ref{oss f_I} -- and call $f$ its shtuka function. We can extend $\phi$ to ideals, sending $J=(a,b)< A$ to the unique monic generator $\phi_J$ of the left ideal $(\phi_a,\phi_b)<\K[\tau]$, following a construction of Hayes (see \cite{Hayes}).

\begin{Def}
Fix $\omega\in \Sf(\phi)$, and for all nonzero ideals $J< A$ define $\chi_\phi(J):=\frac{\phi_J(\omega)}{\omega}$. 
\end{Def}
\begin{oss}
The previous definition does not depend on the choice of $\omega$. For any $a\in A$ and any nonzero ideal $J< A$, since $\phi_{aJ}=\phi_J\circ\phi_a$ we have that
\[\chi_\phi(aJ)=\frac{\phi_{aJ}(\omega)}{\omega}=\frac{\phi_J\circ\phi_a(\omega)}{\omega}=\frac{\phi_J((1\otimes a)\omega)}{\omega}=(1\otimes a)\frac{\phi_J(\omega)}{\omega}=\chi_\phi(a)\chi_\phi(J).\]
It's easy to check that we can extend $\chi_\phi$ to all fractional ideals in a unique way such that for all $a\in K$ and for all fractional ideals $J$ we have $\chi_\phi(a)\chi_\phi(J)=\chi_\phi(aJ)$.
\end{oss}
\begin{prop}\label{divisor chi_I}
For all nonzero ideals $J< A$, $\chi_\phi(J)$ is a rational function on $X_\K$ with sign $1$, and $\Div(\chi_\phi(J))=V_{\bar{I}-\bar{\Omega}-\bar{J}}+J-V_{\bar{I}-\bar{\Omega}}-\deg(J)\infty$.
\end{prop}
\begin{proof}
Consider $H^0(X_\K\setminus\{\infty\},V_{\bar{I}-\bar{\Omega}})=\bigcup_{k\geq0}H^0(X_\K,V_{\bar{I}-\bar{\Omega}}+k\infty)$, which admits as a flag base $\{f\cdots f^{(k)}\}_{k\geq-1}$. For a fixed non principal ideal $J=(a,b)< A$, if we write $\phi_J=\sum_{i=0}^{\deg(J)} (c_i\otimes1)\tau^i$, we get:
\[\chi_\phi(J)=\frac{\phi_J(\omega)}{\omega}=\sum_{i=0}^{\deg(J)} c_i f\cdots f^{(i-1)}\in H^0(X_\K,V_{\bar{I}-\bar{\Omega}}+\deg(J)\infty).\]
Since $c_{\deg(J)}f\cdots f^{(\deg(J)-1)}$ is the summand of highest degree at $\infty$, and since $c_{\deg(J)}=1$ and $\sgn(f)=1$, we have:\[\sgn(\chi_\phi(J))=c_{\deg(J)}\sgn(f)\cdots\sgn(f)^{(\deg(J)-1)}=1.\]
Moreover, if we write $\phi_J=\psi_1\circ\phi_a+\psi_2\circ\phi_b$ for some $\psi_1,\psi_2\in \K[\tau]$, we get:
\[\chi_\phi(J)=\frac{\phi_J(\omega)}{\omega}=\frac{\psi_1\circ\phi_a(\omega)+\psi_2\circ\phi_b(\omega)}{\omega}=(1\otimes a)\frac{\psi_1(\omega)}{\omega}+(1\otimes b)\frac{\psi_2(\omega)}{\omega}.\]
Since $1\otimes a,1\otimes b\in H^0(X_\K\setminus\{\infty\},-J)$, $\frac{\psi_1(\omega)}{\omega},\frac{\psi_2(\omega)}{\omega}\in H^0(X_\K\setminus\{\infty\},V_{\bar{I}-\bar{\Omega}})$, and the degree of $\chi_\phi(J)$ is $\deg(J)$, we get $\chi_\phi(J)\in H^0(X_\K,V_{\bar{I}-\bar{\Omega}}-J+\deg(J)\infty)$. The divisor $D:=V_{\bar{I}-\bar{\Omega}}+\deg(J)\infty-J$ has degree $g$ and is such that:
\begin{align*}
    &[D-D^{(1)}]=[V_{\bar{I}-\bar{\Omega}}-V_{\bar{I}-\bar{\Omega}}^{(1)}]=[\Xi-\infty]&\red([D-g\infty])=\bar{I}-\bar{\Omega}-\bar{J}.
\end{align*}
By Proposition \ref{con}, $D\sim V_{\bar{I}-\bar{\Omega}-\bar{J}}$, and $h^0(V_{\bar{I}-\bar{\Omega}-\bar{J}})=1$, hence the divisor $\Div(\chi_\phi(J))$ is equal to $V_{\bar{I}-\bar{\Omega}-\bar{J}}+J-V_{\bar{I}-\bar{\Omega}}-\deg(J)\infty$.
\end{proof}

Let's include an easy Lemma.

\begin{lemma}\label{ker phi_J} For any nonzero ideal $J<A$, as a $\F_q$-linear endomorphism of $\C$, $\ker(\phi_J)=\exp_\phi(\rho_I IJ^{-1})$.
\end{lemma}
\begin{proof}
We can fix two generators $a,b$ of $J$. By definition of $\phi_J$, there are $\psi_1,\psi_2$ in $\K[\tau]$ such that $\phi_J=\psi_1\circ\phi_a+\psi_2\circ\phi_b$. Since for any $x\in\rho_I IJ^{-1}$, $ax$ and $bx$ belong to $\rho_I I=\ker(\exp_\phi)$, we have:
\[\phi_J\circ\exp_\phi(x)= \psi_1\circ\phi_a\circ\exp_\phi(x)+ \psi_2\circ\phi_b\circ\exp_\phi(x)= \psi_1\circ\exp_\phi(ax)+
\psi_2\circ\exp_\phi(bx)=0.\]
In particular, $\exp_\phi(\rho_I IJ^{-1})\subseteq \ker(\phi_J)$. On the other hand, $\exp_\phi(\rho_I IJ^{-1})$ is isomorphic as an $\F_q$-vector space to $\rho_I IJ^{-1}/\rho_I I$, which has cardinality $q^{\deg(J)}$. Since $\phi_J$ is a polynomial of degree $q^{\deg(J)}$, we get the equality $\exp_\phi(\rho_I IJ^{-1})=\ker(\phi_J)$.
\end{proof}

\begin{prop}\label{red_u(chi)}
For all nonzero ideals $J<A$, $\red_u\left(\chi_\phi(J)(\Xi)\right)=1$.
\end{prop}
\begin{proof}
If $J=(a)$, where $a\in A\setminus\{0\}$ has sign $1$, $\red_u\left(\chi_\phi(J)(\Xi)\right)=\red_u(a\otimes1)=1$. In particular, for any nonzero ideal $J$ and for all $a\in A\setminus\{0\}$,$\red_u\left(\chi_\phi(J)(\Xi)\right)$ is equal to £$\red_u\left(\chi_\phi(aJ)(\Xi)\right)$.
If we write $\phi_J=\sum_{i=0}^{\deg(J)}c_i\tau^i$, we have:
\[c_0=\left(\sum_{i=0}^{\deg(J)}(c_i\otimes1)f\cdots f^{(i-1)}\right)(\Xi)=\chi_\phi(J)(\Xi)\text{ and }c_{\deg(J)}=1.\] By Lemma \ref{ker phi_J}, $\ker(\phi_J)=\exp_\phi(\rho_I IJ^{-1})$; let's fix a set $\{\alpha_i\}_i\subseteq IJ^{-1}$ such that the elements $\{\exp_\phi(\rho_I\alpha_i)\}_i$ are representatives for the quotient $(\ker(\phi_J)\setminus\{0\})/\F_q^\times$. We have:
\[\chi_\phi(J)(\Xi)=c_0=\prod_{\beta\in\ker(\phi_J)\setminus\{0\}}\beta=\prod_{\lambda\in\F_q^\times}\prod_i\lambda\exp_\phi(\rho_I\alpha_i)=\prod_i-(\exp_\phi(\rho_I\alpha_i))^{q-1}.\]
Let's write $\exp_\phi=\sum_{i\geq0}e_i\tau^i$; by Corollary \ref{pi in K} $\rho_I^{q-1}\in\K$, so by Lemma \ref{log_L} the coefficients of $\log_\phi$ are in $\K$, hence the same holds for the $e_i$'s. In particular, for all $i\in\mathbb{N}$ the element $\gamma_i:=\frac{\exp_\phi(\rho_I\alpha_i)}{\rho_I}=\alpha_i\sum_j e_j (\rho_I\alpha_i)^{q^j-1}$ is contained in $\K$, which means that $\red_u(\gamma_i^{q-1})=1$. Since the cardinality of $\{\alpha_i\}_i$ is $\frac{q^{\deg(J)}-1}{q-1}$, we have:
\begin{align*}
    \red_u\left(\chi_\phi(J)(\Xi)\right)&=\prod_i-\red_u\left(\rho_I^{q-1}\right)\red_u(\gamma_i)^{q-1}=\red_u\left(-\rho_I^{q-1}\right)^\frac{q^{\deg(J)}-1}{q-1}\\
    &=\red_u\left(-\rho_I^{q-1}\right)^{\deg(J)}.
\end{align*}
For $d\gg0$, we can pick $a,b\in A\setminus\{0\}$ with $d=\deg(a)=\deg(b)-1$. We have:
\begin{align*}
    1&=\frac{\red_u\left(\chi_\phi(bJ)(\Xi)\right)}{\red_u\left(\chi_\phi(aJ)(\Xi)\right)}=\frac{\red_u\left(-\rho_I^{q-1}\right)^{\deg(bJ)}}{\red_u\left(-\rho_I^{q-1}\right)^{\deg(aJ)}}=\red_u\left(-\rho_I^{q-1}\right)^{\deg(b)-\deg(a)}\\
    &=\red_u\left(-\rho_I^{q-1}\right),
\end{align*}
hence $\red_u\left(\chi_\phi(J)(\Xi)\right)=1$.
\end{proof}

\begin{oss}\label{proportionality of periods}Following a known construction due to Hayes, for any nonzero ideal $J<A$ we denote by $\phi^J:A\to\K[\tau]$ the unique ring homomorphism such that for all $a\in A$ $\phi^J_a\circ\phi_J=\phi_J\circ\phi_a$. As easily shown in \cite{Goss}[Subsection 4.9], the morphism $\phi^J:A\to\K[\tau]$ is a normalized Drinfeld module of rank $1$ and its associated lattice is $\chi_\phi(J)(\Xi)\rho_I J^{-1}I$.\end{oss}

\begin{Def}\label{def anderson}
    Call $a_I\in I$ the unique nonzero element of least degree with sign $1$. The \emph{Anderson zeta function} relative to the Drinfeld module $\phi$ is defined as:
    \[\xi_\phi:=\sum_{J< a_I^{-1}I}\frac{\chi_\phi(J)}{\chi_\phi(J)(\Xi)}\in \K\hat\otimes K.\]
\end{Def}
\begin{oss}\label{oss anderson}
    In the previous definition, $\xi_\phi$ only depends on the ideal class $\bar{I}$ of $I$.
    We can also write $\xi_\phi=a_I\otimes a_I^{-1}\sum_{J<I}\frac{\chi_\phi(J)}{\chi_\phi(J)(\Xi)}$.
\end{oss}
In the notation of Lemma \ref{lemma h_I,J}, we finally state and prove the main theorem of this section.
\begin{teo}\label{xi equation}
The Anderson zeta function $\xi_{\phi}$ is a well defined element of $\K\hat\otimes K$. Moreover, if we fix an ideal $J<A$ such that $J\Omega\cong A$, the following identity holds:
\[\xi_{\phi}\omega_{\phi,J}=(\tilde{\pi}_\phi\otimes1)\sum_{\sigma\in\G(H/K)}h_{J,\bar{A}}^\sigma.\]
\end{teo}
\begin{proof}
Let's fix representatives $J_i< I$ for each ideal class 
$\bar{J_i}\in Cl(A)$. To prove convergence we rearrange the terms:
\begin{align*}
    &\sum_{\tilde{J}<I}\frac{\chi_\phi(\tilde{J})}{\chi_\phi(\tilde{J})(\Xi)}
    =\sum_i\sum_{\substack{\tilde{J}<I\\\tilde{J}\cong J_i}}\frac{\chi_\phi(\tilde{J})}{\chi_\phi(\tilde{J})(\Xi)}=\sum_i\sum_{\substack{a\in J_i^{-1}I\setminus\{0\}\\\sgn(a)=1}}\frac{\chi_\phi(a J_i)}{\chi_\phi(a J_i)(\Xi)}\\
    =&-\sum_i\sum_{a\in J_i^{-1}I\setminus\{0\}}\frac{\chi_\phi(a J_i)}{\chi_\phi(a J_i)(\Xi)}=-\sum_i\left(\frac{\chi_\phi(J_i)}{\chi_\phi(J_i)(\Xi)}\sum_{a\in J_i^{-1}I\setminus\{0\}}\frac{\chi_\phi(a)}{\chi_\phi(a)(\Xi)}\right)\\
    =&-\sum_i\frac{\chi_\phi(J_i)}{\chi_\phi(J_i)(\Xi)}\zeta_{J_i^{-1}I}.
\end{align*}
Since $\omega_{\phi,J}$ is defined as $-\frac{(\tilde{\pi}_\phi\otimes1)h_{J,\bar{I}}}{\zeta_I}(a_I^{-1}\otimes a_I)$, we get: 
\[\xi_\phi\omega_{\phi,J}=(\tilde{\pi}_\phi\otimes 1)\sum_i\frac{\chi_\phi(J_i)}{\chi_\phi(J_i)(\Xi)}\frac{\zeta_{J_i^{-1}I}}{\zeta_I}h_{J,\bar{I}}.\] 
For all $i$, by Proposition \ref{quotient of zetas} $\frac{\zeta_{J_i^{-1}I}}{\zeta_I}\in H(X)$; the evaluation at $\Xi$ of the rational function $\frac{\chi_\phi(J_i)}{\chi_\phi(J_i)(\Xi)}$ is $1$, and by Proposition \ref{divisor chi_I} and Remark \ref{Hayes} its divisor is $H$-rational, hence the function is $H$-rational; finally, $h_{J_i,\bar{I}}\in H(X)$ by Lemma \ref{lemma h_I,J}. We deduce that for all $i$ the summand $\frac{\chi_\phi(J_i)}{\chi_\phi(J_i)(\Xi)}\frac{\zeta_{J_i^{-1}I}}{\zeta_I}h_{J,\bar{I}}$ belongs to $H(X)$, and by Proposition \ref{divisor chi_I} and Lemma \ref{lemma h_I,J} its divisor is:
\begin{align*}
    &(V_{\bar{I}-\bar{J_i},*}^{(1)}-V_{\bar{I},*}^{(1)}-J_i)+(V_{\bar{I}+\bar{J}-\bar{J_i}}+J_i-V_{\bar{I}+\bar{J}})+\\
    +&(V_{\bar{I},*}^{(1)}+V_{\bar{I}+\bar{J}}-J-\Xi)-(2g-\deg(J)-1)\infty\\
    =&V_{\bar{I}-\bar{J_i},*}^{(1)}+V_{\bar{I}+\bar{J}-\bar{J_i}}-J-\Xi-(2g-\deg(J)-1)\infty=\Div(h_{J,\bar{I}-\bar{J_i}}),
\end{align*} 
hence $\frac{\zeta_{J_i^{-1}I}}{\zeta_I}\frac{\chi_\phi(J_i)}{\chi_\phi(J_i)(\Xi)}h_{J,\bar{I}}=(\alpha_i\otimes1)h_{J,\bar{I}-\bar{J_i}}$ for some $\alpha_i\in H^\times$. Since $\sgn(h_{J,\bar{I}-\bar{J_i}})=1$, $\alpha_i$ is equal to the sign of the summand. For all $i$ denote $\rho_{J_i^{-1}I}:=\chi_\phi(J)(\Xi)\rho_I$: By Remark \ref{proportionality of periods}, if $\psi$ is the unique normalized Drinfeld module of rank $1$ whose period lattice $\Lambda_\psi$ is isomorphic to $J_i^{-1}I$, then $\Lambda_\psi=\rho_{J_i^{-1}I}J_i^{-1}I$. If we denote $s_i:=\chi_\phi(J_i)(\Xi)^{-1}\zeta_I^{-1}\zeta_{J_i^{-1}I}$, by Proposition \ref{a_I/pi} we have:
\begin{align*}
    &\frac{s_i^{(1)}}{s_i}=\frac{\left(\chi_\phi(J_i)(\Xi)^{-1}\zeta_{J_i^{-1}I}\right)^{(1)}}{\chi_\phi(J_i)(\Xi)^{-1}\zeta_{J_i^{-1}I}}\cdot\frac{\zeta_I}{\zeta_I^{(1)}} =\frac{\left(\rho_{J_i^{-1}I}^{-1}\zeta_{J_i^{-1}I}\right)^{(1)}}{\rho_{J_i^{-1}I}^{-1}\zeta_{J_i^{-1}I}}\cdot\frac{\rho_I^{-1}\zeta_I}{\left(\rho_I^{-1}\zeta_I\right)^{(1)}}=\left(\frac{f_{\bar{I},*}}{f_{\bar{I}-\bar{J_i},*}}\right)^{(1)}\\
    \Rightarrow&\sgn(s_i)^{q-1}=\sgn\left(\frac{s_i^{(1)}}{s_i}\right)=\frac{\sgn(f_{\bar{I},*})^q}{\sgn(f_{\bar{I}-\bar{J},*})^q}=1;
\end{align*}
in particular, \[\alpha_i=\sgn(s_i)\cdot\sgn(\chi_\phi(J_i))\cdot\sgn(h_{J,\bar{I}})=\sgn(s_i)\in\F_q.\]
On the other hand, by Proposition \ref{quotient of zetas} and Proposition \ref{red_u(chi)} we have:
\[\red_u(\sgn(s_i))=\red_u(\sgn(\chi_\phi(J_i)(\Xi)))^{-1}\cdot\red_u\left(\sgn\left(\frac{\zeta_{J_i^{-1}I}}{\zeta_I}\right)\right)=1,\]
hence $\alpha_i=1$.
We can rewrite:
\[\xi_\phi\omega_{\phi,J}=(\tilde{\pi}_\phi\otimes1)\sum_i h_{J,\bar{I}-\bar{J_i}}=(\tilde{\pi}_\phi\otimes1)\sum_{\sigma\in\G(H/K)}h_{J,\bar{A}}^\sigma. \tag*{\qedhere}\]
\end{proof}
\begin{oss}
Evaluating at $\Xi$ we get $\xi_{\bar{I}}(\Xi)=\# Cl(A)$ modulo the characteristic of $\F_q$.
\end{oss}

\section{Acknowledgements}
Thanks to Nathan Green for his availability to discuss the early drafts of this work. 

\noindent
Thanks to Quentin Gazda, Andreas Maurischat, Tuan Ngo Dac, and Dinesh Thakur for the interesting conversations. 

\noindent
Thanks to Federico Pellarin, for his guidance and patience, both of which have been crucial for the writing of this paper.

\noindent
The author is also thankful to the Department of Mathematics Guido Castelnuovo, where he has carried out his Ph.D. studies. This work has been produced as part of the Ph.D. thesis of the author.

\printbibliography

\end{document}